\documentclass[10pt]{amsart}

\usepackage[all]{xy}                        %

\CompileMatrices                            

\UseTips                                    

\input xypic
\usepackage[bookmarks=true]{hyperref}       

\usepackage{amssymb,latexsym,amsmath,amscd}
\usepackage{xspace}
\usepackage{enumerate}
\usepackage{graphicx}
\usepackage{vmargin}
\usepackage{todonotes}

\usepackage{dcpic,pictexwd}
\usepackage{tcolorbox}
\usepackage{subcaption}

\DeclareMathOperator{\Res}{Res}

\reversemarginpar

\theoremstyle{plain}
\newtheorem{theorem}{Theorem}[section]
\newtheorem*{theorem*}{Theorem}
\newtheorem{proposition}[theorem]{Proposition}
\newtheorem{corollary}[theorem]{Corollary}
\newtheorem{lemma}[theorem]{Lemma}

\theoremstyle{definition}
\newtheorem{definition}[theorem]{Definition}

\newtheorem{notation}[theorem]{Notation}

\newtheorem{remark}[theorem]{Remark}

\newcommand{\enm}[1]{\ensuremath{#1}}          %

\newcommand{\cal}[1]{\mathcal{#1}}

\renewcommand{\bar}[1]{\overline{#1}}

\newcommand{\NN}{\enm{\mathbb{N}}}
\newcommand{\RR}{\enm{\mathbb{R}}}

\newcommand{\ZZ}{\enm{\mathbb{Z}}}
\newcommand{\FF}{\enm{\mathbb{F}}}

\newcommand{\PP}{\enm{\mathbb{P}}}

\newcommand{\Bb}{\enm{\cal{B}}}
\newcommand{\Cc}{\enm{\cal{C}}}

\newcommand{\Ii}{\enm{\cal{I}}}

\newcommand{\Oo}{\enm{\cal{O}}}

\newcommand{\Ss}{\enm{\cal{S}}}
\newcommand{\Tt}{\enm{\cal{T}}}

\newcommand{\Vv}{\enm{\cal{V}}}

\newcommand{\Xx}{\enm{\cal{X}}}

\renewcommand{\phi}{\varphi}
\renewcommand{\theta}{\vartheta}
\renewcommand{\epsilon}{\varepsilon}

\begin{document}

\title{{Twistor fibers in Hypersurfaces of the Flag Threefold}}

\author[A. Altavilla]{Amedeo Altavilla}\address{Dipartimento di Matematica,
  Universit\`a degli Studi di Bari `Aldo Moro', via Edoardo Orabona, 4, 70125,
  Bari, Italia}\email{amedeo.altavilla@uniba.it}

\author[E. Ballico]{Edoardo Ballico}\address{Dipartimento Di Matematica,
  Universit\`a di Trento, Via Sommarive 14, 38123, Povo, Trento, Italia}
\email{edoardo.ballico@unitn.it}

\author[M. C. Brambilla]{Maria Chiara Brambilla}\address{
Universit\`a Politecnica delle Marche, via Brecce Bianche, I-60131 Ancona, Italia}
\email{m.c.brambilla@univpm.it}

\thanks{All the authors are partially supported by GNSAGA. The first named author is partially supported by the INdAM project `Teoria delle funzioni ipercomplesse e applicazioni'.}



\begin{abstract} 
{We study surfaces of bidegree $(1,d)$ contained in the flag threefold {in relation to} the twistor projection.
In particular, we focus on the number and the arrangement of twistor fibers contained in such surfaces.
First, we prove that there is no
irreducible surface of bidegree $(1,d)$ containing $d+2$ twistor fibers {in general position}.
{On the other hand, given any collection of $(d+1)$ twistor fibers satisfying a mild natural constraint, we prove the existence of a surface of bidegree $(1,d)$ that contains them.}
 We improve our results for  $d=2$ or $d=3$, by removing all the generality hypotheses.}
\end{abstract}

\maketitle
\setcounter{tocdepth}{1} 
\setcounter{tocdepth}{2} 

\tableofcontents

\section{Introduction}

The study of the twistor geometry of the flag {threefold} is motivated by the search for Riemannian $4$-manifolds admitting several integrable complex structures 
compatible with the prescribed metric (see e.g.~\cite{fujpon,pontecorvomatann,salviac}). 
{In this context, a recent trend is the investigation
of specific cases} in order to find explicit examples~\cite{altbalagag,altsarf,armpovsal,chirka,gensalsto}.

The flag threefold $\FF$ can be seen as the twistor space  
of 
the complex projective plane $\mathbb{P}^{2}$
endowed with all its standard structures~\cite{AHS, Hitchin}, 
$$
\pi:\FF\to\PP^{2}.
$$
{The threefold $\FF$ embeds in $\PP^{2}\times\PP^{2}$, hence it is possible to define a natural notion of bidegree 
$(d_{1},d_{2})$ for {curves and surfaces}
in $\FF$.}

In \cite{abbs} we have started a detailed analysis of the geometry of the algebraic curves and surfaces contained in {the flag threefold} $\FF$, in relation to the twistor projection.
In particular,  twistor fibers are smooth {irreducible} curves of bidegree (1,1). In~\cite{abb} we gave a first bound on the maximum number of smooth {irreducible curves of bidegree} $(1,1)$ contained in a smooth surface $S\subset\FF$. 
Here, 
by focussing our attention on a particular family of surfaces, we obtain {stronger and more
significant} results. In fact,
we analyse the case of bidegree $(1,d)$ surfaces in $\FF$ and study the number and the 
arrangements of twistor fibers contained in them.

{A surface $S$ of bidegree $(1,d)$  
can be realised as a $d : 1$ branched cover of $\PP^{2}$, or alternatively as a blow up of $\PP^{2}$ in $n = 1+d+d^{2}$ points. If $d = 1$ then $S$ is a (toric) del Pezzo surface of degree 6, and this case was studied in~\cite{abbs}. 
In the present paper we give results for arbitrary $d$, as well as focussing on $d = 2$ and $d = 3$.
In general, if $S$ has bidegree $(d_{1},d_{2})$, the twistor projection $\pi$ restricted to $S$ is a branched cover of degree $d_{1}+d_{2}$ of $\PP^{2}$, so if
$(d_{1},d_{2})=(1,d)$, it is natural to compare our results with those obtained for degree $d+1$ surfaces in $\PP^{3}$ viewed as the
twistor space of the $4$-sphere~\cite{Hitchin}. In the latter case, surfaces of degree $2$ and $3$ have been studied in some detail.
In particular, surfaces of degree $2$ in $\PP^{3}$ might contain $0,1$ or $2$
twistor fibers and if a surface of degree $2$ contains more than $2$ twistor fibers, then it contains infinitely many of them~\cite{salviac}. This is completely analogous to what we found in~\cite{abbs} for surfaces of bidegree $(1,1)$ in $\FF$.
On the contrary, for $d=2$ we observe a  difference between the two cases.
Indeed, for degree $3$ surfaces in $\PP^{3}$, the maximum number of twistor fibers contained in a smooth surface is $5$~\cite{armpovsal}, while we prove here that this number reduces to $4$ for smooth surfaces of bidegree $(1,2)$ in $\FF$.}

{In order to present our results, we need to introduce some definitions.}
The flag threefold can be explicitly defined as $$\FF:=\{(p,\ell)\in\PP^{2}\times\PP^{2}\,|\,p\ell=0\},$$ where $p=[p_{0}:p_{1},p_{2}],\ell=[\ell_{0}:\ell_{1}:\ell_{2}]\in\PP^{2}$ and $p\ell=p_{0}\ell_{0}+p_{1}\ell_{1}+p_{2}\ell_{2}$. The notation $(p,\ell)$, would recall a couple \textit{(point,line)}, 
and the condition $p\ell=0$ translates as \textit{$p$ belongs to $\ell$}. In order to simplify the notation, 
we identify the second factor $\PP^{2\vee}$ with $\PP^{2}$.  

We have three projection maps: $\pi_{1},\pi_{2},\pi:\FF\to\PP^{2}$, defined as
$$
\pi_{1}(p,\ell)=p,\quad \pi_{2}(p,\ell)=\ell,\quad\pi(p,\ell)=\bar p\times\ell={[\bar p_{1}\ell_{2}-\bar p_{2}\ell_{1},\bar p_{2}\ell_{0}-\bar p_{0}\ell_{2},\bar p_{0}\ell_{1}+\bar p_{1}\ell_{0}],}
$$
where the third one is the twistor map. The fibers of such three maps are the object of our investigation.
 In particular, for any $q\in\PP^{2}$,  the three fibers
 $\pi_{1}^{-1}(q), \pi_{2}^{-1}(q)$ and $\pi^{-1}(q)$ are curves of bidegree $(0,1), (1,0)$ and $(1,1)$, respectively.
While the fibers of $\pi_{1}$ and $\pi_{2}$ exhaust the family of bidegrees $(1,0)$ and $(0,1)$ curves,
twistor fibers are a (non-open Zariski dense) subset of those of bidegree $(1,1)$. 
{It was shown in~\cite[Section 3.1]{abbs} that any bidegree $(1,1)$ curve has Hilbert polynomial equal to $2t+1$
with respect to the standard Segre embedding of $\PP^{2}\times\PP^{2}$; for this reason we will call them  \textit{conics}.}
There are only two types of bidegree $(1,1)$ curves: the reducible ones (union of a bidegree 
$(1,0)$ and of a bidegree $(0,1)$ curves intersecting at a point), and the smooth ones. All of them can be described as
$$
L_{q,m}:=\{(p,\ell)\in\FF\,|\,pm=0,\,q\ell=0\},
$$
where $q,m\in\PP^{2}$ and, the reducible and smooth cases are obtained for $qm=0$ or $qm\neq 0$, respectively.

In twistor theory, an important role is played by an antiholomorphic involution without fixed points, which identifies twistor fibers~\cite{AHS,Hitchin}. In our case, this map can be defined as $j:\FF^{2}\to\FF^{2}$, where
$$j(p,\ell)=(\bar\ell,\bar p).$$
A smooth conic $L_{q,m}$ is a twistor fiber if and only if $j(L_{q,m})=L_{q,m}$ if and only if $m=\bar q$.
{Moreover, it is natural to classify objects in $\FF$ up to projective automorphisms  coming from the lift via $\pi$ of a unitary automorphism of $\PP^{2}$, i.e. a holomorphic isometry with respect to the Fubini-Study metric.
Such transformations of $\FF$ are exactly those which commute with $j$ (see~\cite[Lemma 5.4]{abbs} for the flag manifold case).
Thus, in particular, the number of twistor fibers contained in a given surface and their arrangement are unitary invariants.}
 
In order to state our main results we need some more notation.
We denote by $\Cc(1)=\Cc$ the set of smooth conics in $\FF$ and by $\Cc(n)$, $n\ge 2$, the set of
$n$ pairwise disjoint smooth conics. In an analogous way we define $\Tt(1)=\Tt\subset\Cc$ as the set of twistor
fibers and $\Tt(n)\subset\Cc(n)$, $n\ge 2$, as the set of $n$ pairwise disjoint twistor fibers.

We will see in Remark~\ref{LR} that for any couple of different smooth conics, there is a unique bidegree $(1,0)$ curve $L=\pi_{2}^{-1}(q_{2})$
and a unique bidegree $(0,1)$ curve $R=\pi_{1}^{-1}(q_{1})$ such that $L$ and $R$ intersect both smooth conics. In the case of 
a couple of twistor fibers we also have $R=j(L)$. 
{We say that three or more smooth conics are {\it collinear} if there is a $(1,0)$ curve $L$ which intersects all of them. To be collinear, for three or more smooth conics,
is a Zariski closed condition. }

To be more precise, 
in Definition \ref{definitionC*}, we define the set 
{ $\Cc^*(n)$ which parametrizes all $A\in \Cc(n)$ such that $\#(L\cap A) \le 2 $ for all curves $L$ of bidegree $(1,0)$. Clearly $\Cc^*(1)= \Cc(1)$, and $\Cc^*(2)= \Cc(2)$, while for $n\ge3$
the open set $ \Cc^*(n)$ is given by the set of disjoint smooth conics such that no three of them are collinear.
}
Moreover, we set $\Tt^*(n):= \Tt(n)\cap \Cc^*(n)$. 
In Theorem~\ref{c01} we characterize the elements $A\in\Cc^{*}(d+1)$  to be those which do not
obstruct the linear system $|\Ii_{A}(1,d)|$. 

We now summarize the main results of the paper.
In Section~\ref{Section1d}, we study surfaces of bidegree $(1,d)$ containing a certain number of smooth conics or twistor fibers, and we prove 
the following two theorems.
\begin{theorem}\label{u6}
For any $d\in \NN$ and  $A\in \Tt^*(d+2)$, there is no irreducible surface of bidegree $(1,d)$ containing $A$. 
\end{theorem}

\begin{theorem}\label{new-thm}
Fix integer {$d\ge 1$} and $0\le n \le {d+2}$. There is an irreducible   $S\in |\Oo_{\FF}(1,d)|$ containing exactly $n$ twistor fibers. 
\end{theorem}

{
We also show, in Theorem~\ref{u5}, that 
the first result is sharp. Indeed,} for any $A\in\Tt^{*}(n)$, with
$0\le n\le d+1$, we are able to prove the existence of an irreducible surface of
bidegree $(1,d)$ containing $A$ and no other twistor fibers. This last issue requires
 some effort and the proof is divided in {a detailed analysis} of several cases.

{Theorem \ref{new-thm} is a consequence of Theorem \ref{u5} and Theorem~\ref{nok1}.
More precisely, in Theorem \ref{u5}, for $0\le n\le d+1$,  we prove that,
fixed any union $A$ of $0\le n \le d+1$ non-three-by-three collinear twistor fibers, there is an irreducible $(1,d)$-surface containing $A$ and no other twistor fibers.
The extremal case $n=d+2$ is considered in Theorem~\ref{nok1}, where we prove that,
given $d+2$ general collinear twistor fibers, there is an irreducible surface of bidegree $(1,d)$ containing them.}

In Section \ref{sez-4}, we focus on surfaces of bidegree $(1,2)$ and $(1,3)$. 
The main  results are summarized by the following statements:
\begin{theorem}\label{i2i1}
Fix $0\le n \le 3$. There is a smooth $S\in |\Oo_{\FF}(1,2)|$ containing exactly $n$ twistor fibers. 
Moreover, there exists a bidegree $(1,2)$ irreducible surface containing exactly $4$ twistor fibers.
\end{theorem}
\begin{theorem}\label{n6}
There is no irreducible $S\in |\Oo_{\FF}(1,2)|$ containing at least ${5}$ twistor fibers.
\end{theorem}

\begin{theorem}\label{n7}
There is no irreducible $S\in |\Oo_{\FF}(1,3)|$ containing at least $6$ twistor fibers.
\end{theorem}

{The first existence result (Theorem \ref{i2i1}) follows from Theorem \ref{terzo-I}, for $0\le n\le 3$, and
Theorem~\ref{aaa1}, for the case $n=4$. 
In the extremal case $n=4$, we will also show that the surfaces are singular along a line.}

{The two non-existence results (Theorems \ref{n6} and \ref{n7})
are proved in the last Section \ref{last-subsection}. } An essential tool is Lemma~\ref{bo5} which states that if a surface of bidegree $(1,d)$ contains $d+3$ or more {collinear}
twistor fibers, then this surface is reducible and one of its components is a surface of bidegree $(1,1)$ containing $4$ of the
prescribed twistor fibers.

{The proofs of our results are based on cohomological methods and on a careful description of the linear systems of surfaces in the flag threefold.}

\section{Preliminaries and first results}\label{prelimi}

In this section we collect some known results about algebraic curves and surfaces in the flag threefold. Then we give first results on the space of bidegree $(0,d)$ and $(1,d)$-surfaces containing a certain number of twisted fibers. In particular, we introduce the concept of \textit{collinear} smooth conics and give a topological characterization in terms of the cohomology of certain ideal sheaves.

For most of the known material on $\FF$ we refer to~\cite{abbs} and to~\cite[Section 2]{abb}.
However, in order to be as self-contained as possible, we will recall some basic ideas and results here.

Consider the multi projective space $\PP^{2}\times \PP^{2}$; an element $(p,\ell)\in\PP^{2}\times \PP^{2}$ will be a pair written in the following form $p=[p_{0}:p_{1}:p_{2}]$, 
$\ell=[\ell_{0}:\ell_{1}:\ell_{2}]^{\top}$, so that
$p\ell=p_{0}\ell_{0}+p_{1}\ell_{1}+p_{2}\ell_{2}$.
Even though it is classically embedded in $\PP^{2}\times \PP^{2\vee}$, we can see $\FF:=\{(p,\ell)\in\PP^{2}\times\PP^{2}\,|\,p\ell=0\}$ as a hypersurface of bidegree $(1,1)$ of $\PP^2\times \PP^{2}$.  
We denote by $\Pi_1$ and $\Pi_2$ the two standard projections of $\PP^2\times \PP^{2}$ and use lower case for their restrictions, i.e. $\pi _i =\Pi_{i|\FF}$, $i=1,2$.
Thus, the two natural projections define a natural notion of bidegree for algebraic surfaces in $\FF$. Furthermore, for all $(a,b)\in \ZZ^2$ we have the following natural exact sequence
\begin{equation}\label{eqpp1}
0 \to \Oo_{\PP^2\times \PP^{2}}(a-1,b-1)\to \Oo_{\PP^2\times \PP^{2}}(a,b)\to \Oo_{\FF}(a,b)\to 0,
\end{equation}and, for any $(a,b)\in \NN^2$,  we get (see e.g.~\cite[Lemma 2.3]{abbs})
\begin{equation}\label{equation1}
h^0(\Oo_{\FF}(a,b)) =
\frac{(a+1)(b+1)(a+b+2)}{2}\qquad\mbox{ and }\qquad h^1(\Oo_{\FF}(a,b)) =0.
\end{equation}

It will be useful to recall from~\cite[Proposition 3.11]{abbs} the
multiplication rules in the Chow ring: 
\begin{equation}\label{tablechow}
\begin{split}
\Oo_{\FF}(1,0)\cdot\Oo_{\FF}(1,0)\cdot\Oo_{\FF}(1,0)=0,&\qquad
\Oo_{\FF}(1,0)\cdot\Oo_{\FF}(0,1)\cdot\Oo_{\FF}(1,0)=1,\\
\Oo_{\FF}(0,1)\cdot\Oo_{\FF}(1,0)\cdot\Oo_{\FF}(0,1)=1,&\qquad
\Oo_{\FF}(0,1)\cdot\Oo_{\FF}(0,1)\cdot\Oo_{\FF}(0,1)=0.
\end{split}
\end{equation}

\subsection{Curves in $\FF$ and smooth conics}\label{curvesandsurfaces}

Let us recall the notion of bidegree for the family of algebraic curves in $\FF$ already given in~\cite{abbs,abb}.

\begin{definition}\label{bidegree}
Let $C\subset\FF$ be an irreducible and reduced algebraic curve. We define the bidegree of $C$
as the couple of positive integers $(d_{1},d_{2})$, where
$d_{i}=0$ if $\pi_{i}(C)=\{x\}$, otherwise $d_{i}=\deg(\pi_{i}(C))\deg({\pi_{i}}_{|C})$.

If a curve $D$ has irreducible components $C_{1},\dots, C_{s}$ then the bidegree of $D$ is 
the sum of the bidegrees of $C_{1},\dots , C_{s}$.
\end{definition}

Recall from~\cite[Remark 2.4]{abb} that if a curve $C$ is such that $C\cdot \Oo_\FF(1,0)=d_1$ and $C\cdot \Oo_\FF(0,1)=d_2$, then
it has bidegree $(d_1,d_2)$.

From the multiplication table \eqref{tablechow} we can easily derive the following formula.
\begin{lemma}\label{uc1}
For any choice of non-negative integers $a,b,c,d$, the one-dimensional cycle $\Oo_{\FF}(a,b)\cdot \Oo_{\FF}(c,d)$ has bidegree 
$$(ad+b(c+d),a(c+d)+bc).$$
\end{lemma}

\begin{proof}
We have $\Oo_{\FF}(a,b)\cdot \Oo_{\FF}(c,d) = ac \Oo_{\FF}(1,0)\cdot \Oo_{\FF}(1,0) + (ad+bc) \Oo_{\FF}(1,0)\cdot \Oo _{\FF}(0,1)+bd\Oo_{\FF}(0,1)\cdot \Oo_{\FF}(0,1)$. So the thesis is easily obtained by recalling that $ \Oo_{\FF}(1,0)\cdot \Oo_{\FF}(1,0)$ (resp. $\Oo_{\FF}(1,0)\cdot \Oo _{\FF}(0,1)$, resp. $\Oo_{\FF}(0,1)\cdot \Oo_{\FF}(0,1)$) is a one-dimensional cycle of bidegree $(0,1)$ (resp. bidegree $(1,1)$, resp. bidegree $(1,0)$).
\end{proof}

\begin{remark}\label{rem1001}
Note that the fibers of $\pi_{1}$ are algebraic curves of bidegree $(0,1)$,
while those of $\pi_{2}$ have bidegree $(1,0)$ (see, e.g.~\cite[Section 3]{abbs}). Moreover, all bidegree $(0,1)$ curves can be seen as complete intersections of two different $(1,0)$-surfaces (and analogously for bidegree $(1,0)$ curves).
\end{remark}

Among all the algebraic curves in $\FF$, we focus our attention
on the family of bidegree $(1,1)$ curves. These are described geometrically in~\cite[Section 3.1]{abbs} and are parameterized by $(q,m)\in\PP^{2}\times\PP^{2}$. In fact, as anticipated in the introduction, each of these curves can be written as
$$
L_{q,m}:=\{(p,\ell)\in\FF\,|\,p\in m,\,\ell\ni q\}=\{(p,\ell)\in\FF\,|\,q\ell=0,\,pm=0\}.
$$
 There are two types of these curves:
the smooth and irreducible ones (when $q m\neq 0$) and the union of a $(1,0)$ and
of a $(0,1)$ intersecting at one point (when $q m=0$, i.e. $(q,m)\in\FF$). In any case, any curve of bidegree $(1,1)$ can be seen as the complete intersection of a surface of bidegree $(1,0)$ with one of bidegree $(0,1)$. {As already mentioned in the introduction}, the $4$-dimensional family of smooth irreducible $(1,1)$ curves is denoted by $\Cc$. The elements of $\Cc$ are called \textit{smooth conics}. 
\begin{remark}\label{remC}
From the  definition of smooth conics, 
it is clear that, for any $C\in\Cc$ we have that $\pi_{i}(C)$ is
a line in $\PP^{2}$.
\end{remark}

\begin{remark}\label{LR}
Note that for any two different elements $L_{q,m}, L_{q',m'}\in\Cc$ there exists a unique curve $L$ of bidegree $(1,0)$ and a unique curve $R$ of bidegree $(0,1)$ such that $L$ and $R$ meets both
$L_{q,m}$ and $L_{q',m'}$ {at one point} (see Figure~\ref{fig1}). 
From the analysis made in~\cite[Section 3.1]{abbs} it is easy to see that $L=\pi_{2}^{-1}(q\times q')$ and $R=\pi_{1}^{-1}(m\times m')$, where $\times$ stands for the standard (formal) cross product. Equivalently, $L = \pi _2^{-1}(\mathrm{Sing}(\pi _2(A))$ and $R = \pi _1^{-1}(\mathrm{Sing}(\pi _1(A))$.
\begin{figure}[h]
\vspace{-5pt}
\includegraphics[scale=0.2]{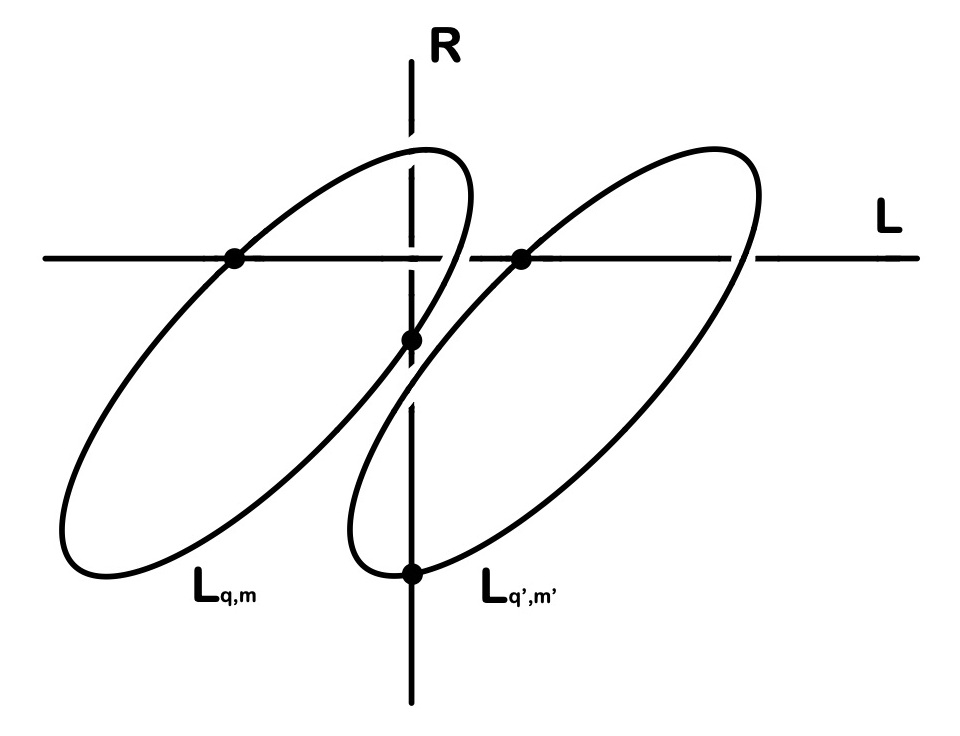}
\vspace{-5pt}
\caption{Any two smooth conics are {connected by}  a curve of bidegree $(1,0)$ and by a curve of $(0,1)$.}\label{fig1}
\end{figure}
We say that three or more disjoint smooth conics are {\it collinear} if they intersect the same $(1,0)$ curve $L$. 
\end{remark}

%
%

The fibers  of the twistor projection $\pi:\FF\to\PP^{2}$ (see~\cite[Section 5]{abbs}) form a subset $\Tt$ of the family of conics $\Cc$.  The twistor fibers are also characterized to
be the irreducible elements in $\Cc$ that are fixed by the anti-holomorphic involution $j:\FF\to\FF$ defined as
$$j(p,\ell)=(\bar\ell,\bar p).$$
Being the set of fixed points of $j$, a curve $L_{q,m}$ belongs to $\Tt$ if and only if $m=\bar q$. Furthermore, the set $\Tt$ is a Zariski dense in $\Cc$ (see, e.g.~\cite[Section 4]{abb}).
\begin{remark}\label{LjR}
If $L$ is the bidegree curve $(1,0)$ connecting two different twistor
fibers (see Remark~\ref{LR}), then the curve of bidegree $(0,1)$ connecting them is exactly
$R=j(L)$. {Thus, if three twistor fibers are collinear, then they intersect the same $(1,0)$ curve and the same $(0,1)$ curve.} 
\end{remark}

Recall from the introduction that, for any positive integer $n$, $\Cc(n)$ denotes the $4n$-dimensional set of $n$ pairwise disjoint elements of $\Cc$ and $\Tt(n)$ denotes the set of $n$ pairwise disjoint elements of $\Tt$. As before, $\Tt(n)$ is a Zariski dense in $\Cc(n)$ (see again~\cite[Section 4]{abb}). 

We now introduce the following crucial definition.
\begin{definition}\label{definitionC*}
For any $n\ge 1$ let $\Cc^*(n)$ be the set of all $A\in \Cc(n)$ such that
for any curve $L$ of bidegree $(1,0)$, it holds $\#(L\cap D) \le 2 $. Set $\Tt^*(n):= \Tt(n)\cap \Cc^*(n)$. 
\end{definition}
Obviously we have $\Cc^*(n)= \Cc(n)$ and $\Tt^*(n)= \Tt(n)$, for $n=1,2$.
{For $n\ge3$ the set $\Cc(n)\setminus \Cc^*(n)$ parameterizes unions of $n$ disjoint smooth conics such that at least three of them are collinear,}
hence $\Cc^{*}(n)$ is an open Zariski dense in $\Cc(n)$, as well as $\Tt^{*}(n)$ in $\Tt(n)$. 
Therefore for any $n\ge 1$, all of the following inclusions are Zariski dense: $\Tt^{*}(n)\subset\Tt(n)\subset\Cc(n)$,
$\Tt^{*}(n)\subset\Cc^{*}(n)\subset\Cc(n)$.

\subsection{Surfaces of bidegree $(1,0)$ and $(0,1)$}\label{sechirz}
Now we turn our attention back to surfaces. We recall from~\cite[Section 3.2]{abbs} and~\cite[Section 2]{abb} that $(1,0)$ and $(0,1)$-surfaces are
Hirzebruch surfaces of the first type. In particular, a surface $X$ of bidegree $(1,0)$ can be seen as the lift, via $\pi_{1}$, of a line (and analogously for a surface $Y$ of bidegree $(0,1)$). Using this description, it is easy to see that each of these surfaces represents the blow up of $\PP^{2}$ at one point.
Let $F_{1}$ be a Hirzebruch surface of type $1$; we now describe the relation between the generators of the Picard group of $F_{1}$ and the family of curves in $\FF$ described earlier. We recall that
$\mathrm{Pic}(F_1)=\ZZ h\oplus\ZZ f$, where 
$$h^2=-1,\qquad f^2=0,\qquad hf=1.$$
\begin{notation}
For the following analysis and the rest of the paper, $X$ will denote a surface of bidegree $(1,0)$, while $Y$ a surface of bidegree $(0,1)$.
\end{notation}

If we identify a surface $X$ with $F_{1}$, we get that  
$\Oo_X(1,0)\simeq\Oo_{F_1}(f)$, which again corresponds to the set of curves in $\FF$ of bidegree (0,1) contained in $X$. On the other hand, we have that $\Oo_X(0,1)\simeq \Oo_{F_1}(h+f)$, which corresponds to the elements of $\Cc$.

Thus, for any $a,b\in\ZZ$ and for any $\alpha,\beta\in\ZZ$, we obtain the following two relations
\begin{equation}\label{Hirz}
\Oo _X(a,b) \cong \Oo_{F_1}(bh+(a+b)f),\quad \mbox{ and }\quad  \Oo_{F_1}(\alpha h+\beta f)\cong \Oo _X(\beta-\alpha ,\alpha).\end{equation}

For a surface $Y$ of bidegree $(0,1)$ we can derive similar formul\ae:
\begin{equation}\label{coeff}\Oo _Y(a,b) \cong \Oo_{F_1}(ah+(a+b)f),\quad \mbox{ and }\quad
  \Oo_{F_1}(\alpha h+\beta f)\cong \Oo _Y(\alpha ,\beta-\alpha),\end{equation}
  
 for any $a,b\in\ZZ$ and  for any $\alpha,\beta\in\ZZ$.

\begin{remark}\label{m3}
Let $X$ be a surface of bidegree $(1,0)$. Then $X$ contains no element of $\Cc(2)$. In fact, every element of $\Cc$ in $X$ corresponds to an element of $\Oo_{F_{1}}(h+f)$ and any two elements of type $h+f$ meet.
The same holds for surfaces $Y$ of bidegree $(0,1)$.
{In particular, for every surface of bidegree $(1,0)$ or $(0,1)$ there is exactly one twistor fiber contained in it.}
\end{remark} 

We now recall from~\cite[Lemma 2.5]{abb} that, for any $a,b\ge0$, 
by using the exact sequence
$$0 \to \Oo_\FF(a-1,b)\to \Oo_\FF(a,b)\to \Oo_X(a,b)\to 0,$$
and its analog for $Y$, we have that
$$h^0(\Oo_X(a,b))=a(b+1)+
\binom{b+2}{2},\quad
h^0(\Oo_Y(a,b))=b(a+1)+
\binom{a+2}{2},$$
while
$$h^1(\Oo_X(a,b))=h^1(\Oo_Y(a,b))=0.$$ 
Furthermore, if $a>0,b>0$, the line bundles $\Oo_X(a,b)$ and $\Oo_Y(a,b)$ are very ample.

\subsection{Surfaces of bidegree $(0,d)$ and $(1,d)$}\label{0d1d}

We now move on to studying {higher} bidegree surfaces. 
We start with some considerations about bidegree $(0,d)$
surfaces.

\begin{remark}\label{m3.1}
As described in~\cite[Section 3.3]{abbs}, every irreducible and reduced surface $S$ of bidegree $(0,d)$ is equal to $\pi_{1}^{-1}(C)$ for some irreducible and reduced curve $C$ of degree $d$. Therefore, for $d\ge 2$, no irreducible and reduced $S\in|\Oo_{\FF}(0,d)|$ contains a smooth conic. Otherwise, thanks to Remark~\ref{remC}, $\pi_{i}(S)$ would contain a line, but $\pi_{1}(S)$ is an irreducible and reduced curve of degree $d$.
\end{remark}

For $n\ge2$ we now compute how many bidegree $(0,d)$-surfaces contain a fixed element of $\Cc(n)$. First we set up the following notation.
\begin{notation}
We will denote by $\Ii_{U,V}$ the ideal sheaf of a scheme $U$ contained in a projective variety $V$; whenever $V=\FF$ we will omit it. In particular, if $A\in\Cc(n)$ we will write $\Ii_{A}:=\Ii_{A,\FF}$.
\end{notation}

\begin{lemma}\label{c02}
Fix $d\ge 0$, $n\ge 1$, 
and $A\in \Cc(n)$. We have 
$$h^0(\Ii_{A}(0,d)) =\frac{(d-n+2)(d-n+1)}{2}$$ 
and
$$h^1(\Ii _{A}(0,d)) = \begin{cases}
\frac{n(n-1)}{2}&\mbox{ if }n\le d+1 \\
n(d+1)-\frac{(d+2)(d+1)}{2}&\mbox{ if }n\ge d+1. 
\end{cases}$$
\end{lemma}

\begin{proof}
Recall that $\Oo _{\FF}(0,d) = \pi _1^\ast(\Oo_{\PP^2}(d))$, and 
$\Ii _{A}(0,d) = \pi _1^\ast(\Ii_{T,\PP^2}(d))$
where $T=\pi_1(A)$ is a union of $n$ distinct lines in $\PP^2$.
In general, we have that:
$$h^0(\Oo_{\FF}(0,d)) =\binom{d+2}{2},\quad h^0(\Ii_{A}(0,d)) =\binom{d+2-n}{2},\quad h^0(\Oo_{A}(0,d)) =n(d+1),$$
hence, using the exact sequence
$$0\to \Ii _{A}(0,d)\to \Oo _{\FF}(0,d)\to \Oo _{A}(0,d)\to0,$$
since $h^1(\Oo_{\FF}(0,d)) =0$ and $\binom{d+2-n}{2}=0$ if $n\ge d+1$, we get the result.
\end{proof}

\begin{remark}\label{postc02}
As a direct consequence of the previous lemma, we can state that, for any 
$C\in\Cc$, there is only one surface $Y$ in $|\Ii_{C}(0,1)|$ and, analogously,
only one surface $X$ in $|\Ii_{C}(1,0)|$.
\end{remark}

\begin{remark}\label{remarkCHI}
If $A\in\Cc(n)$, we consider the following exact sequence:
$$0\to \Ii _{A}(a,b)\to \Oo _{\FF}(a,b)\to \Oo _{A}(a,b)\to0.$$
Since the conics in $A$ are all disjoint, for each $(a,b)\in\NN^2$ we have that
\begin{equation}
h^0(\Oo_{A}(a,b))=n(a+b+1),\label{h0O}\end{equation}
(see, e.g.~\cite[Sequence (7) and proof of Theorem 4.4]{abb}).
Recall from Formula~\eqref{equation1} that $h^0(\Oo_{\FF}(a,b)) =
\frac{(a+1)(b+1)(a+b+2)}{2}$. Therefore, as $h^{1}(\Oo _{A}(a,b))=h^1(\Oo_{\FF}(a,b))=0$, for any $A\in \Cc(n)$ and $d\ge0$, we have
\begin{equation}\label{p1} 
\chi(\Ii_A(1,d))=h^{0}(\Ii_A(1,d))-h^{1}(\Ii_A(1,d))
=(d+1)(d+3)-n(d+2).
\end{equation}
Hence, if $n=d+2$, we have $h^0(\Oo _A(1,d))=(d+2)^{2}$ and so
\begin{equation}\label{eq000}
h^0(\Oo_{\FF}(1,d)) =(d+1)(d+3) = (d+2)^2-1
= h^0(\Oo _A(1,d))-1
\end{equation}
and,  if $n=d+1$: 
$$h^0(\Oo_{\FF}(1,d)) =h^0(\Oo_A(1,d)) +d+1.$$ 
\end{remark}

In what follows, we are going to implicitly use the following simple observation.
\begin{remark}\label{m1}
Since the elements of $\Tt(1)$ are the fibers of the twistor map, each $p\in \FF$ is contained in a unique element, $C$, of $\Tt$. Since $j(C)=C$, $j(p)\in C$. {Analogously}, note that $p$ is contained in a unique curve of bidegree $(1,0)$ and a unique curve of bidegree $(0,1)$, $\pi _2^{-1}(\pi _2(p))$ and $\pi _1^{-1}(\pi _1(p))$. 
\end{remark}

The following remark will be used several times in the next pages.

\begin{remark}\label{m0000}
Fix a positive integer $d$ and an irreducible $S\in |\Oo_{\FF}(1,d)|$. Since $\pi _{1|S}$ is birational onto its image, $S$ is rational. By B\'ezout Theorem, for any $p\in \PP^2$,  the bidegree $(1,0)$ curve $\pi _2^{-1}(p)$ is either contained in $S$ or intersects $S$ in a single point (scheme-theoretically).
\end{remark}

We conclude this subsection with a technical result that will be used in the next pages.
In~\cite[Proposition 4.1]{abb} we proved that there exists at most one surface of bidegree $(a,b)$ containing a number equal to or greater than $a^{2}+ab+b^{2}$ of smooth conics. We will now generalize this result to a more general context.

\begin{proposition}\label{c0} 
Fix $(a,b,c,d)\in \NN^4$ so that  $(a,b)\ne (0,0)$, $c>0$, $d>0$. Take $A\in \Cc(n)$. Assume the existence of an irreducible and reduced $S\in |\Oo_{\FF}(a,b)|$ containing $A$
and assume one of the following conditions:
\begin{enumerate}
\item $ad+b(c+d)<n$;
\item $a(c+d)+bc< n$;
\item $ad+b(c+d) = a(c+d)+bc=n$.
\end{enumerate}
Then every element of $|\Ii_A(c,d)|$ contains $S$ and in particular $c\ge a$ and $d\ge b$.
\end{proposition}

\begin{proof}
Assume by contradiction the existence of $S'\in |\Ii_A(c,d)|$ such that $S'\nsupseteq S$. We have 
\begin{align*}
\Oo_{\FF}(a,b)\cdot \Oo_{\FF}(c,d) =& ac \Oo_{\FF}(1,0)\cdot \Oo_{\FF}(1,0) + (ad+bc) \Oo_{\FF}(1,0)\cdot \Oo _{\FF}(0,1)\\
&+bd\Oo_{\FF}(0,1)\cdot \Oo_{\FF}(0,1).
\end{align*}

Since $S'\nsupseteq S$,  the intersection $S\cap S'$ is a curve of bidegree $(ad+b(c+d),a(c+d)+bc)$ (see Lemma~\ref{uc1}).
Since $c>0$ and $d>0$, then $\Oo_{\FF}(c,d)$ is ample. Moreover, since $S$ is irreducible, $A$ has bidegree $(n,n)$, and $A$ is not connected, then the intersection contains {some more components in addition to } $A$. So either $ad+bc+bd >n$ and $ac+ad+bc\ge n$ or $ad+bc+bd \ge n$ and $ac+ad+bc> n$. 
\end{proof}

%
%
%

\subsection{Non-collinear smooth conics}

We now want to characterize the conics in $\Cc^*(n)$ in terms of cohomology. We start by showing that the vanishing of certain cohomology groups implies that an element $A\in\Cc(n)$ lies in $\Cc^*(n)$.

\begin{lemma}\label{lemmino}
Fix $d\ge 0$, $3\le n\le d+1$ and  $A\in \Cc(n)$. If $h^1(\Ii _A(1,d)) = 0$ then $A\in \Cc^*(n)$.
\end{lemma}
\begin{proof}
Suppose, by contradiction, that there exists
a curve $L$ of bidegree $(1,0)$ such that $\#(L\cap A) \ge 3$ 
and consider the exact sequence that defines $\Ii_{A}$:
\begin{equation}\label{sequenza}0\to \Ii _{A}(1,d)\to \Oo _{\FF}(1,d)\to \Oo _{A}(1,d)\to0.
\end{equation}
We will prove that
the restriction map $H^0(\Oo_{\FF}(1,d))\to H^0(\Oo_A(1,d))$ is not surjective, which implies that
$h^1(\Ii _{A}(1,d)) >0$.
Let us assume that it is surjective. Then consider the following diagram
$$
\begindc{\commdiag}
\obj(50,25)[SA]{$\Oo_\FF(1,d)$}
\obj(50,10)[CP]{$\Oo_L(1,d)$}
\obj(75,25)[H]{$\Oo_{A}(1,d)$}
\obj(75,10)[T]{$\Oo_{L\cap A}(1,d)$}
\mor{SA}{CP}{ }[\atright,\solidarrow]
\mor{SA}{H}{ }[\atleft,\solidarrow]
\mor{CP}{T}{ }[\atright,\solidarrow]
\mor{H}{T}{ }[\atleft,\solidarrow]
\enddc\,.
$$
Since the vertical maps are surjective,
then the induced map $H^0(\Oo_{\FF}(1,d))\to H^0(\Oo_{A\cap L}(1,d))$ is also surjective and thus of rank at least 3 (since $L\cap A$ has cardinality at least $3$ and the irreducible components of $A$ are pairwise disjoint).
On the other hand, the restriction map $H^0(\Oo_{\FF}(1,d))\to H^0(\Oo_L(1,d))$ has rank $2$, and this leads to a contradiction.
\end{proof}

Before completing the characterization of conics in $\Cc^*(d+1)$, we will introduce a general construction that will be used in several subsequent discussions.
Let $A\in\Cc(n)$ and let $C$ be any connected component of $A$. Set $B:= A\setminus C$.
Then, for any $a,b\ge0$, if $Y\in|\Oo_{C}(0,1)|$, we have the following residual exact sequence 
$$
0\to\Ii_{Res_{Y}(A)}(a,b-1)\to\Ii_{A}(a,b)\to\Ii_{A\cap Y,Y}(a,b)\to 0,
$$
but since $Res_{Y}(A)=B$ and $A\cap Y=(B\cap Y)\cup C$, we have
\begin{equation}\label{successione}
0\to\Ii_{B}(a,b-1)\to\Ii_{A}(a,b)\to\Ii_{(B\cap Y)\cup C,Y}(a,b)\to 0.
\end{equation}

Obviously, an analogous sequence can be written for $X\in|\Oo_{C}(1,0)|$.

\begin{theorem}\label{c01}
Fix $d\ge 0$ and $A\in \Cc(d+1)$. We have $A\in \Cc^*(d+1)$ if and only if $h^1(\Ii _A(1,d)) = 0$.
\end{theorem}
\begin{proof}
Thanks to the previous lemma we only need to prove that $A\in \Cc^*(d+1)$ satisfies $h^1(\Ii _A(1,d)) = 0$. We use induction on $d\ge 0$.
The case $d=0$ is true by Lemma \ref{c02}. So we can assume $d>0$ and use induction on $d$. Let $C$ be a connected component of $A$, set $B:= A\setminus C$ and call $Y$ the unique element of $|\Ii_C(0,1)|$. Consider the residual exact sequence \eqref{successione}, with $a=1$ and $b=d$.
%
Since $A\in \Cc(d+1)$, $C\cap B =\emptyset$ and $B\cap Y$ is a set of $d$ different points, up to the identification of $D$ with $F_1$ we have $$\Ii_{(B\cap Y)\cup C,Y}(1,d) \cong \Ii_{(B\cap Y)\cup C,F_1}(1,d) (h+(d+1)f)\cong \Ii_{B\cap Y,F_1}(df)\,.$$ 

Using \eqref{successione} and induction, we are left to prove that $h^1(\Ii_{B\cap Y,F_1}(df))=0$ if and only if $A\in \Cc^*(d+1)$. 

Consider now the following exact sequence
\begin{equation}
0 \to \Ii_{B\cap Y,F_1}(df)\to \Oo_{F_1}(df)\to \Oo_{B\cap Y}(df) \to 0.
\end{equation}

Since $h^0(\Oo_{F_1}(df))=d+1$ and $h^0(\Oo_{B\cap Y}(df))=d$,  we have $h^1(\Ii_{B\cap Y,F_1}(df))>0$ if and only if 
$h^0(\Ii_{B\cap Y,F_1}(df))\ge2$.
The last inequality means that there are at least two different sets of $d$ fibers containing the set of $d$ points $B\cap Y$. This is equivalent to the fact that there exists a fiber $L\in |f|$ such that $\#(B\cap L)\ge 2$. 
 Since $L$ is a curve of bidegree $(1,0)$ in $\FF$, then 
 $L\cdot Y =0$ in the intersection ring of $\FF$, so $L\subset Y$, and we get $L\cap C\ne \emptyset$. 
 Thus $\#(L\cap A)\ge 3$, which means that $A\not\in \Cc^*(d+1)$. 
\end{proof}

\begin{corollary}\label{cor1}
Fix $d\ge 0$, $0\le n\le d+1$ and  $A\in \Cc^*(n)$. 
Then $h^1(\Ii _A(1,d)) = 0$. In particular, for $n=0,1,2$ {and for any $A\in \Cc(n)$,} we have 
$$h^1(\Ii _A(1,1)) = 0 \quad \text{ and } \quad
h^{0}(\Ii _A(1,1))=8-3n.$$
\end{corollary}
\begin{proof}
{By using \cite[Remark 4.3]{abb}, we easily get the first part of the statement.
The second part follows from Formula~\eqref{p1}.}
\end{proof}

We note that since $\Tt(n)$ is a Zariski dense of $\Cc(n)$, then the characterization given by Theorem \ref{c01} also holds for the set $\Tt^*(n)$.

\begin{lemma}\label{ii1}
Fix an integer $d\ge 0$ and $A\in \Cc^*(d+2)$. Then $h^1(\Ii_A(1,d))\le 1$.
\end{lemma}

\begin{proof}
The lemma is true for $d=0$, because $h^0(\Ii_A(1,0)) =0$ (see Remark~\ref{m3}). 
We assume $d>0$ and use induction on $d$. Let $C$ be a connected component of $A$, set $B:= A\setminus C$ and call $Y$ the unique element of $|\Ii_C(0,1)|$. Consider the residual exact sequence \eqref{successione}, with $a=1$ and $b=d$.
Since $A\in \Cc(d+2)$, $C\cap B =\emptyset$ and $B\cap Y$ is formed by $d+1$ different points, up to the identification of $D$ with $F_1$ we have $$\Ii_{(B\cap Y)\cup C,Y}(1,d) \cong \Ii_{(B\cap Y)\cup C,F_1}(1,d) (h+(d+1)f)\cong \Ii_{B\cap Y,F_1}(df)\,.$$ 

Using \eqref{successione} and induction, we are left to prove that $h^1(\Ii_{B\cap Y,F_1}(df))=0$ if $A\in \Cc^*(d+2)$. 

Consider now the following exact sequence
\begin{equation}
0 \to \Ii_{B\cap Y,F_1}(df)\to \Oo_{F_1}(df)\to \Oo_{B\cap Y}(df) \to 0\,.
\end{equation}

Since $h^0(\Oo_{F_1}(df))=d+1$ and $h^0(\Oo_{B\cap Y}(df))=d+1$,  we have $h^1(\Ii_{B\cap Y,F_1})>0$ if and only if 
$h^0(\Ii_{B\cap Y,F_1}(df))>0$. This is equivalent to the fact that there exists a fiber $L\in |f|$ such that $\#(B\cap L)\ge 2$. 
 Since $L$ is a curve of bidegree $(1,0)$ in $\FF$, then 
 $L\cdot Y =0$ in the intersection ring of $\FF$, therefore $L\subset Y$, and so we get $L\cap C\ne \emptyset$. 
 Thus $\#(L\cap A)\ge 3$, which means that $A\not\in \Cc^*(d+2)$. 
\end{proof}

As said in Remark~\ref{LR}, for any element $A\in\Cc(2)$ there exists a unique
curve $L$ of bidegree $(1,0)$ and a unique curve $R$ of bidegree $(0,1)$
such that both intersect the elements of $A$ {at a point}. As described in the following result, it turns out that {$A \cup L\cup R$} is the base locus
of $|\Ii_A(1,1)|$.

%
%

\begin{proposition}\label{ee1} 
For any $A\in \Cc(2)$, we have that 
\begin{enumerate}
\item the general element in $|\Ii_A(1,1)|$ is irreducible;
\item the base locus $\Bb$ of $|\Ii_A(1,1)|$ is $A\cup L\cup R$, where $L$  and  $R$ are the curves described in Remark~\ref{LR}.
\end{enumerate}
\end{proposition}

\begin{proof}
Since $A\in\Cc(2)$, by Corollary~\ref{cor1} we have $h^1(\Ii_A(1,1)) =0$ and $h^0(\Ii_A(1,1)) =2$. 
Let $C_{1}$ and $C_{2}$ be the connected components of $A$. Denote by $X_{i}$ the only element of $|\Oo_{\FF}(1,0)|$ containing $C_{i}$ and by $Y_{i}$ the only element of $|\Oo_{\FF}(0,1)|$ containing $C_{i}$. 
The surfaces $X_{1}\cup Y_{2}$ and $X_{2}\cup Y_{1}$ are the only reducible elements of $|\Ii_A(1,1)|$ and hence, the general element in $|\Ii_A(1,1)|$ is irreducible and \textit{(1)} is proved. 

To prove \textit{(2)} we analyze the base locus $\Bb$ of $|\Ii_A(1,1)|$. If $S,S'\in |\Ii_A(1,1)|$ are irreducible and $S\ne S'$, then the one-dimensional cycle $S\cap S'$ has bidegree $(3,3)$ and it contains $A$, which has bidegree $(2,2)$. Let $R:= X_{1}\cap X_{2}$ and $L:= Y_{1}\cap Y_{2}$, where $X_{i}$ and $Y_{i}$ are the surfaces defined in the first part of this proof. The curves $L$ and $R$ are exactly those given in Remark~\ref{LR}. In particular,
$\#(L\cap A)=\#(R\cap A)=2$ and hence, by B\'ezout and Remark~\ref{m0000}, $L\cup R\subset\Bb$. Furthermore, {recall that the reducible surfaces, $X_{1}\cup Y_{2}$ and $X_{2}\cup Y_{1}$ belong to $|\Ii_A(1,1)|$ and their intersection $(X_{1}\cup Y_{2})\cap(X_{2}\cup Y_{1})$
is $A\cup L\cup R$. So the base locus of $|\Ii_A(1,1)|$ is exactly $\Bb=A\cup L\cup R$.}
\end{proof}

\begin{remark}
By generalizing the proof of Proposition~\ref{ee1}, we can say something about the base locus 
of $\Ii_A(1,d)$, for $A\in \Cc(n)$.
Fix the integers $d>0$ and $n\ge 2$ and take any $A\in \Cc(n)$. Then, the base locus $\Bb$ of $\Ii_A(1,d)$ contains all curves $L$ of bidegree $(1,0)$ such that $\#(L\cap A)\ge 2$. If $A\in \Cc^*(n)$, then there are exactly $\binom{n}{2}$ such curves $L$ (the number of lines connecting two points in a set of $n$ general points).
If $A\in \Tt(n)$, then $\#(j(L)\cap A)\ge 2$ for all  $L$ {such that $\#(L\cap A)\ge 2$}, since $j(A)=A$. 
\end{remark}

\section{Surfaces of bidegree $(1,d)$}\label{Section1d}

In this section we prove Theorems~\ref{u6} and~\ref{new-thm}. In particular, we give some results for the case of a surface of bidegree $(1,d)$. Later,
in the following section, we will specialize to the cases $d=2$ and $d=3$. The case $d=1$ has been studied in great detail in~\cite{abbs}, and here we add a simple lemma useful for what follows.

\begin{lemma}\label{primo-caso}
For any $A\in \Tt^*(3)$, there is no irreducible surfaces of bidegree $(1,1)$ containing $A$. 
\end{lemma}
\begin{proof}
Suppose there exists an irreducible surface $M$ of bidegree $(1,1)$ containing $A\in \Tt^*(3)$.
First of all, thanks to~\cite[Corollary 8.4]{abbs}, we have that $M=j(M)$.
Then, as explained at the beginning of Section 7.1 in~\cite{abbs}, either $M$ is smooth or
reducible. But if $M$ is a smooth $j$-invariant surface of bidegree $(1,1)$ containing 
$3$ twistor fibers, then it contains infinitely many of them, and these are parameterized by a
circle (see~\cite[Theorem 7.2]{abbs}). However, smooth surfaces of bidegree $(1,1)$ can be seen as the blow up of $\PP^{2}$ at three points both via $\pi_{1}$ and $\pi_{2}$.
In particular, up to unitary transformations, it is possible to write $M$ as the set 
$\{([p_{0}:p_{1}:p_{2}],[\ell_{0}:\ell_{1}:\ell_{2}])\in\FF\,|\,p_{1}\ell_{1}+\lambda p_{2}\ell_{2}\}$,
with $\lambda\in\RR\setminus\{0,1\}$. In these coordinates, $M$ contains 
$\pi_{\mu}^{-1}([1:0:0]), \pi_{\mu}^{-1}([0:1:0]), \pi_{\mu}^{-1}([0:0:1])$, for $\mu=1,2$ and, the  family of twistor fibers $\pi^{-1}([q_{0}:q_{1}:q_{2}])$ defined by:
$$
\begin{cases}
q_{0}=0\mbox{ and }|q_{1}|^{2}\lambda+|q_{2}|^{2}=0&\mbox{ if }\lambda<0\,,\\
q_{1}=0\mbox{ and }|q_{2}|^{2}-|q_{0}|^{2}(\lambda-1)=0&\mbox{ if }\lambda>1\,,\\
q_{2}=0\mbox{ and }|q_{1}|^{2}\lambda+|q_{0}|^{2}(\lambda-1)=0&\mbox{ if }0<\lambda<1\,.
\end{cases}
$$
Take for instance $\lambda<0$, then every twistor fiber in $M$ intersects the line $L=\pi_{2}^{-1}([1:0:0])$ of bidegree $(1,0)$. An analogous consideration holds if $0<\lambda<1$ or $\lambda>1$. So we get a contradiction.
\end{proof}

{In the previous lemma we showed that an irreducible $(1,1)$-surface
cannot contain three twistor fibers in general position.
On the other hand if $M$ is a $(1,1)$-surface containing a given $A\in \Tt(3)\setminus \Tt^*(3)$, then by \cite[Corollary 8.3]{abbs} we have that $M$ is $j$-invariant, so it is either smooth or reducible.} Moreover, if $M$ contains infinitely many twistor fibers, then all of them intersect a bidegree $(1,0)$ curve $L$ and its associated $(0,1)$ curve $R=j(L)$.

\begin{remark}\label{esistono}
In~\cite[Section 8.1]{abbs} we gave examples of bidegree $(1,1)$ smooth surfaces containing exactly $0$, $1$ or $2$ twistor fibers.
\end{remark}

\begin{remark}
Since a smooth surface of bidegree $(1,1)$ is a Del Pezzo surface of degree $6$, it is characterized {either by  the three curves of bidegree $(1,0)$ that it contains, or by the three curves of bidegree $(0,1)$ that it contains.}
In fact, recall that these surfaces represent the blow up of $\PP^{2}$ at three points with respect to either $\pi_{1}$ or $\pi_{2}$. Note that if a smooth surface of bidegree $(1,1)$ is $j$-invariant, then it is uniquely determined by three twistor fibers contained in it, and not by the curves $L$ and $R=j(L)$ (of bidegree $(1,0)$ and $(0,1)$, respectively) which intersect all the twistor fibers. 
\end{remark}


Now we are ready to give the proof of our first main theorem.

\begin{proof}[Proof of Theorem \ref{u6}:]
Thanks to Remark~\ref{m3} and Lemma~\ref{primo-caso}, the result is true for $d=0$ and $d=1$.

{Suppose now that $d\ge 2$ and, by contradiction, that
 $S$ is an irreducible $(1,d)$-surface containing $A\in \Tt^*(d+2)$.} Call $C$ a connected component of $A$ and set $B:= A\setminus C$. Let $Y$ be the unique (by {Lemma~\ref{esistono}}) element of $|\Ii_C(0,1)|$ and consider the
 following exact sequence {(which is a special case of the one in formula~\eqref{successione})}:
\begin{equation}\label{equ1}
0 \to \Ii_B(1,d-1)\to \Ii_A(1,d)\to \Ii_{(B\cap Y)\cup C,Y}(1,d) \to 0.
\end{equation}

From Formul\ae~\eqref{h0O} and~\eqref{eq000}, we have $h^0(\Oo_A(1,d)) =(d+2)^2 =h^0(\Oo_{\FF}(1,d)) +1$. 
Obviously we have that $B\cap Y$ is formed by $d+1$ points, and 
up to the identification of $Y$ with $F_1$ given in Formula~\eqref{coeff}, since the curve $C$ corresponds to an element of type $h+f$ in $F_{1}$, we can write $\Ii_{(B\cap Y)\cup C,Y}(1,d) \cong {\Ii_{(B\cap Y)\cup C,F_1}(h+(d+1)f)\cong} \Ii_{B\cap Y,F_1}(df)$. Since $A\in \Tt^*(d+2)$ and every element of $|f|$ meets $C$ (indeed $(h+f)f=1$), the restriction to $B\cap Y$ of the ruling morphism $D\to \PP^1$ associated with $|f|$ is injective. 
Thus, $h^1(Y,\Ii_{A\cap Y,Y}(1,d)) =0$ and the exact sequence \eqref{equ1} gives $h^1(\Ii _B(1,d-1)) \ge h^1(\Ii _A(1,d)) \ge 2$, where the latter is greater than or equal to $2$ because $\chi(\Ii _A(1,d))=-1$ (see Formula~\eqref{p1}), and we assume that $h^0(\Ii _A(1,d)) \ge 1$. 
So we also have $h^0(\Ii_B(1,d-1)) >0$. 

Recall that $B\in \Tt^*(d+1)$ and thus, by the inductive assumption, $B$ is not contained in any irreducible $E\in |\Oo_{\FF}(1,d-1)|$. Therefore, thanks to Remarks~\ref{m3} and~\ref{m3.1}, there must be an irreducible $M\in |\Oo_{\FF}(1,1)|$ containing at least $3$ connected components of $B$, say $B'\subset M$ with $B'$ of bidegree $(3,3)$.
Hence, by Lemma~\ref{primo-caso},
there exists a curve $L$
 of bidegree $(1,0)$ such that $\#(L\cap B')=3$. So $A\notin \Tt^*(d+2)$, a contradiction. 
 \end{proof}

Having proved that an irreducible bidegree $(1,d)$-surface cannot contain $d+2$ (or more) non-collinear twistor fibers, we now prove  that all the other cases can indeed occur.
In particular, in the following result we prove a stronger version of Theorem~\ref{new-thm} in the case of $n\le d+1$.

\begin{theorem}\label{u5}
Fix integers {$d\ge 1$} and $0\le n \le {d+1}$. 
Then, for any $A\in\Tt^*(n)$
there exists an irreducible $S\in |\Oo_{\FF}(1,d)|$ containing $A$. Moreover, the general $S\in |\Ii_{A}(1,d)|$ {contains no other twistor fibers}.
\end{theorem}

\begin{proof}
We use induction on the integer $d$. If $d=1$, the statement is true by Remark~\ref{esistono}.

Now assume $d\ge2$ and take an element $A\in \Tt^*(n)$. Since $\Tt^*(n)$ is Zariski dense in $\Cc(n)$, $A$ has the bigraded Hilbert function of a general  element of $\Cc(n)$. 
Thus, thanks to Corollary \ref{cor1}, we have that $h^1(\Ii_A(1,d)) =0$ and, by Formula~\eqref{p1}
$$
h^{0}(\Ii_A(1,d))=(d+1)(d+3)-n(d+2)=:N_{n}+1.
$$
Fix a connected component $C$ of $A$ and set $B:= A\setminus C$. 
If $Y$ denotes be the only $(0,1)$-surface containing $C$, Corollary \ref{cor1} implies that $h^1(\Ii_B(1,d-1)) =0$ and, again by Formula~\ref{p1}
$$
h^{0}(\Ii_B(1,d-1))=d(d+2)-(n-1)(d+1).
$$

By the inductive assumption, we know that $|\Ii_B(1,d-1)|\neq \emptyset$ and  a general $W\in |\Ii_B(1,d-1)|$ is irreducible. 
Thus $Y\cup W\in |\Ii_A(1,d)|$ and $Y\cup W$ has $2$ irreducible components, one of which has bidegree $(1,d-1)$. 

Let $C_1,\ldots,C_{n}$ be the connected components of $A$, $B_{i}:=A\setminus C_{i}$ and $Y_i$ be the unique element in $|\Ii_{C_i}(0,1)|$.
The set of all the reducible surfaces $W\cup Y_i \in|\Ii_A(1,d)|$, where  $W\in |\Ii_{B_{i}}(1,d-1)|$, 
is the union of $e$ projective spaces (one for each choice of $C_{i}$), each of them of codimension $h^{0}(\Ii_A(1,d))-h^{0}(\Ii_B(1,d-1))=d+2-n>0$
 in $|\Ii_A(1,d)|{=\PP^{N_{n}}}$ 
(of codimension $1$ if $n=d+1$). Therefore, they do not cover all $ |\Ii_A(1,d)|$.

Now we want to exclude other possible splittings. In particular, we consider reducible surfaces of the form $W_1\cup D_1$ with $W_1$ irreducible, $D_1$ possibly reducible of bidegree $(0,x)$ for some $x\ge 2$ and hence $W_1$ of bidegree $(1,d-x)$. Remark \ref{m3.1} shows that only irreducible components of $D_1$ of bidegree $(0,1)$ may contain some component of $A$. We get that the surface is of the form $W_1\cup D_2\cup D_3$ with $W_1\cup D_2$ of bidegree $(1,d-1)$, but, as shown before, these kind of surfaces do not cover all $ |\Ii_A(1,d)|$, hence we have the thesis.

Now we prove that a general $S\in |\Ii_A(1,d)|$ does not contain any other twistor fibers. We start by analyzing the case $d=2$ and {discussing the cases $n=1,2,3$} separately.

Assume $n=1$. Fix $C\in \Tt(1)$. Let $\Cc_C(1)$ denote the set of all $B\in \Cc(1)$ such that $B\cap C=\emptyset$. Note that $\Tt(1)\setminus \{C\} = \Cc_C(1)\cap \Tt(1)$.
For any $B\in \Cc_C(1)$, we have $h^1(\Ii _{C\cup B}(1,2)) =0$ and hence $h^0(\Ii_{C\cup B}(1,2)) =h^0(\Ii_C(1,2))-4$.
Let $\Xx_C$ be the set of all smooth and irreducible surfaces of bidegree $(1,2)$ containing $C$. It is a non-empty Zariski open subset of $|\Ii_C(1,2)|$. For any $B\in \Cc_C(1)$ the set of all $S\in \Xx_C$ containing $B$ has complex codimension $4$
and hence real codimension $8$ as real manifolds. Since $\Tt(1)$ has real dimension $4$, a general $S\in \Xx_C$ contains no other twistor fiber.

Now let  $n=2$. Fix $A\in \Tt(2)$ and let $\Cc_A(1)$ denote the set of all $B\in \Cc(1)$ such that $B\cap A=\emptyset$ and $B\cup A\in \Cc^*(3)$. 
Note that $h^1(\Ii _{A\cup B}(1,2)) =0$. So, as in the previous step, we get that a sufficiently general $S\in |\Ii_A(1,2)|$ does not contain an element of $\Tt_A(1)$. Assume $S$ contains $B\in \Tt(1)$ such that there exists a curve of bidegree $(1,0)$ that intersects every connected component of $A\cup B$. Note that $L$ is uniquely determined by $A$.
Let $\Cc(A,L)$ denote the set of all $B\in \Cc(1)$ such that $B\cap A=\emptyset$ and $L$ meets $B$ and set $\Tt(A,L):= \Cc(A,L)\cap \Tt(1)$. For any $o\in L\setminus (L\cap C)$, the set of all $B\in \Cc(A,L)$ containing $o$ is a non-empty family of complex dimension $1$, while there exists a unique twistor fiber containing $o$. The set $\Cc(A,L)$ is a complex manifold of dimension $2$, while $h^1(\Ii_{A\cup B}(1,2)) =1$ and hence $h^0(\Ii_{A\cup B}(1,2)) =h^0(\Ii_A(1,2)) -3$. Since $\dim \Cc(A,L)=2$, a general $S\in |\Ii_A(1,2)|$ does not contain an element of $\Cc(A,L)$. {Hence, there is} no twistor fiber $B$ such that $A\cup B\in \Tt^*(3)$. 

Assume now that $n=3$. First, consider a general $A\in \Tt^*(3)$. Then we have $h^1(\Ii _A(1,2)) =0$ and $h^0(\Ii_A(1,2)) =3$. For any $x\in \{0,1,2,3\}$, let $\Cc(A,x)$ denote the set of all $C\in \Cc(1)$ such
that $A\cap C=\emptyset$ and $h^0(\Ii_{A\cup C}(1,2)) =x$ and set $\Tt(A,x):= \Cc(A,x)\cap \Tt(1)$. For a general $D\in \Cc(4)$ we have $h^0(\Ii _D(1,1)) =0$ (but $h^1(\Ii_D(1,1)) =1$). 
Fix $C\in \Cc(1)$ so that $C\cap A=\emptyset$, call $Y$ the only element of $|\Ii_C(0,1)|$. 
Since $C\cap A=\emptyset$, no connected component of $A$ is contained in $Y$, and $Y\cap A$ is formed by $3$ points, all of them in $Y\setminus C$.  We can easily handle the case $x=0$, since any curve $C\in \Cc(A,0)$ is not contained in any element of $|\Ii_A(1,2)|$. We have $\Oo_Y(1,2)(-C)\cong \Oo_{F_1}(2f)$ and hence $C\in \Cc(A,0)$ if no curve of bidegree $(1,0)$ $L\in |f|$ intersects $2$ of the components of $A$ {(since $h^1(\Ii_A(1,1)) =1$ the last statement is only an ``if'' and not an ``if and only if'')}. A necessary condition for being $C\in \Cc(A,2)$ is that $L$ intersects all connected components of $A$, but this is excluded
because $A\in \Tt^*(3)$. Since $A\in \Tt^*(3)$, there are exactly $3$ curves $L_1,L_2,L_3$ of bidegree $(1,0)$ intersecting $2$ of the connected components of $A$. We claim that a general $S\in |\Ii_A(1,2)|$ does not contain $C\in \Cc(1)$ such that
$C\cap A=\emptyset$. 
The family of smooth conics intersecting $L_i$ has complex dimension $2$, while
the family of twistor fibers intersecting $L_{i}$ has real dimension $2$. As the general 
$S\in |\Ii_A(1,2)|$ has only finitely many conics, it has only a finite number of elements in 
$\Cc(A,1)$ and, for the general, none of them is a twistor fiber.


We now proceed to analyze the case $d\ge 3$.
Suppose the general surface of $|\Ii_{A}(1,d)|$ contains the  twistor fiber $C\nsubseteq A$. So $A\cap C=\emptyset$. Set $A':= A\cup C$. Take $Y\in |\Oo_Y(0,1)|$ containing $C$ and consider the residual exact sequence
\begin{equation}\label{eqa0}
0 \to \Ii _A(1,d-1) \to \Ii_{A'}(1,d) \to \Ii_{(Y\cap A)\cup C,Y}(1,d)\to 0.
\end{equation}
We have $\Ii_{C,Y}(1,d)\cong \Oo_{F_1}(df)$.

Assume first $n\le d$. If $A'\in \Tt^*(n+1)$, then $h^1(\Ii_A'(1,d)) =0$ and hence $h^0(\Ii_A'(1,d)) =h^0(\Ii_A(1,d))-d-2$. Since $\dim \Cc(1)=4$, we have for $d\ge 3$ that
the general $S\in |\Ii_A(1,d)|$ does not contain $C$, so $A\cup C\in\Tt^*(n+1)$.

Now suppose $A'\notin \Tt^*(n+1)$.  Then there are connected components $C'$ and $C''$ of $A$ such that $C\cup C'\cup C''\notin \Tt^*(3)$, i.e. $C$ intersects the unique line $L$ which meets $C'$ and $C''$. Since $\dim L=1$, to exclude this case it is sufficient to prove that $h^0(\Ii _B(1,d)) \le h^0(\Ii_A(1,d))-2$, i.e. $h^0(F_1,\Ii_{A\cap Y}(df)) \le d$. But since $h^0(\Oo_{F_1}(df)) =d+1$, then $\Oo_{F_1}(df)$ is globally generated and $A\cap Y\ne \emptyset$,  therefore
$h^0(F_1,\Ii_{A\cap Y}(df)) \le d$.

Now let us assume $n=d+1$ and that $A'\in \Cc^*(d+2)$. By Lemma \ref{ii1}  we have $h^1(\Ii_{A'}(1,d)) \le 1$ and hence $h^0(\Ii_{A'}(1,d)) \le h^0(\Ii_A(1,d))-d-1$. 
Now suppose $A'\notin \Cc^*(d+2)$. We need $h^0(F_1,\Ii_{A\cap Y}(df))\le d-1$. Let $\pi: F_1\to \PP^1$ denote the ruling of $F_1$. Since $\Oo_{\PP^1}(d)$ is very ample, $h^0(F_1,\Ii_{A\cap Y}(df))\le d-1$ if and only if $\#\pi(A\cap Y)\ge 2$, which is true because $\#(A\cap Y)=d+1\ge 3$ and (since $A\in \Cc^*(d+1)$ no fiber $F$ of $\pi$ contains at least $3$ points of $A$). The only remaining case is $d=3$ and $n=4$, which can be handled by simply adapting the previous argument for $d=2$ and $n=3$.
\end{proof}

Having proved Theorem~\ref{new-thm} in the case  $n\le d+1$ for non collinear twistor fibers, we now focus on the case of $n=d+2$. In this case we need some further preliminary results.

\begin{lemma}\label{aaa1}
Fix $d>0$, $n\le d+2$ and consider a general $A\in \Tt(n)$. Then $h^1(\Ii_A(1,d)) =0$.
\end{lemma}

\begin{proof}
Since $\Tt(n)$ is Zariski dense in $\Cc(n)$, it is sufficient to prove the statement for a general $A\in \Cc(n)$. 
Obviously, it is sufficient to prove the case $n=d+2$. 
Take a connected component $C$ of $A$ and set $B:= A\setminus C$ and $Y\in |\Ii_C(0,1)|$. 
Consider the residual exact sequence of $Y$, as in Formula~\eqref{equ1}.
Recall the correspondence $Y\cong F_1$ given in Formula~\eqref{coeff} and let $\rho: Y\to \PP^1$ denote its ruling. Since $A$ is general $\#\rho(A\cap Y)=d+1$ and hence $h^i(F_1,\Ii_{B\cap Y}(df)) =0$, for $ i\in\{0, 1\}$. Hence $h^1(\Ii_A(1,d)) =0$.
\end{proof}

{From Theorem \ref{u6}, we already know that an irreducible $(1,d)$ surface can contain a union of $d+2$ twistor fibers only if it belongs to  
$\Tt(d+2)\setminus\Tt^{*}(d+2)$.} 
Therefore, we introduce the following {notation for sets of disjoint smooth conics that are collinear}.
Given a curve $L$ of bidegree $(1,0)$ 
and an integer $n>0$, let $\Cc(n,L)$ denote the set of all $A\in \Cc(n)$ such that every connected component of $A$ meets $L$. The set $\Cc(n,L)$ is isomorphic (as real algebraic variety)
to the set $\Ss(L,n)$ of all subsets of $L$ with cardinality $n$ and is therefore irreducible. An analogous definition and observation can be made for a curve $R$ of bidegree $(0,1)$, and of course, for the family $\Tt$ instead of $\Cc$.

\begin{lemma}\label{no2}
Fix integers $n\ge 3$ and $d\ge 1$ and take any $A\in \Tt(n,L)$. Then,  $$h^1(\Ii _A(1,d)) \ge n-2 +\max \{0,n-(d+1)\}.$$
\end{lemma}

\begin{proof}
%
%
Let $C_1,\dots ,C_n$ be the connected component of $A$.  Let $S_i\subset C_i$ be an arbitrary union of $d+2$ {distinct} points on each conic and $S:= S_1\cup \cdots \cup S_n$. Since $C_i$ is a smooth rational curve, the restriction map $H^0(\Oo_{C_i}(1,d))\to H^0(\Oo_{S_i}(1,d))$
is bijective. So the restriction map $H^0(\Oo_A(1,d))\to H^0(\Oo_S(1,d))$ is bijective and {$\chi(\Oo_A(1,d))=\chi(\Oo_S(1,d))$. Thus we get $\chi(\Ii_{A}(1,d))=\chi(\Ii_{S}(1,d))$. }

{
Since $A\in \Tt(n,L)$, we know that there exists $L$ of bidegree $(1,0)$ which intersects every conic in $A$ (and the $(0,1)$ curve $j(L)$ does the same). So we can choose $S$ such that
 $n$ points are on $L$ and $n$ points are on $j(L)$. In other words we assume $(A\cap L) \cup (A\cap j(L))\subseteq S$.
Using B\'ezout and the fact that the bidegree is $(1,d)$, we get that $(n-2) +\max \{0,n-(d+1)\}$ of these points can be omitted without changing the set $|\Ii_{S}(1,d)|$, and hence $H^0(\Ii_{S}(1,d))$. It follows that $h^1(\Ii _A(1,d)) \ge n-2 +\max \{0,n-(d+1)\}$.}
\end{proof}

\begin{remark}\label{remmmm}
Thanks to the previous lemma, if $A\in\Tt(3,L)$, for some bidegree $(1,0)$ curve $L$,
then $h^1(\Ii _A(1,1)) \ge 2$, and hence $h^0(\Ii _A(1,1)) \ge 1$. However, since no surface of bidegree $(1,0)$ or $(0,1)$ contains an element of $\Tt(2)$, then every $S\in|\Ii_{A}(1,1)|$
is irreducible. Hence, 
Proposition~\ref{c0}  gives that $|\Ii_A(1,1)|=\{S\}$ and so, thanks to Formula~\eqref{p1}, $h^1(\Ii_A(1,1)) =2$.
\end{remark}

{
Finally, the following result completes the proof of Theorem \ref{new-thm}, in the case of $d+2$ twistor fibers.
}


\begin{theorem}\label{nok1}
Fix an integer $d\ge 2$ and take a general $A\in \Tt(d+2,L)$. Then $h^0(\Ii _A(1,d)) \ge d$ and the general $S\in |\Ii_A(1,d)|$ is irreducible.
\end{theorem}

\begin{proof}
Applying Lemma \ref{no2} with $n=d+2$, we get $h^1(\Ii_A(1,d)) \ge d+1$. Since $\chi(\Ii_{A}(1,d))=(d+1)(d+3)-(d+2)^{2}=-1$, we get $h^0(\Ii_A(1,d))\ge d$ and hence, $|\Ii_A(1,d)|\neq \emptyset$.  
 
Now we prove that a general element in $|\Ii_A(1,d)|$ is irreducible.
Take $S\in |\Ii_A(1,d)|$. 
Every surface of bidegree $(1,0)$ or $(0,1)$ contains at most one connected component of $A$. Therefore, $S$ cannot be the union of a surface of 
bidegree $(1,0)$ and $d$ of bidegree $(0,1)$.
No irreducible and reduced surface of bidegree $(0,x)$, for $x\ge 2$, contains a twistor fiber.
If $d=2$, then $h^0(\Ii_A(1,2))\ge 2$. However, thanks to Remark~\ref{remmmm}, for any choice of $C\in A$ there is only
one $M\in|\Ii_{(A\setminus C)}(1,1)|$. {Since we have considered all the possible reducible elements of $|\Ii_A(1,2)|$, we get the thesis.
}

Now suppose $d>2$. Since $A$ has finitely many components, it suffices to prove that
for any $x\in \{3,\dots ,d-1\}$, any union $E$ of $x$ connected components of $A$, and any connected component $C$ of $A\setminus E$ we have 
\begin{equation}\label{ine}
h^0(\Ii_E(1,x-2)) < h^0(\Ii_{E\cup C}(1,x-1)),\end{equation}
and then proceed as in the proof of Theorem~\ref{u5}.
%
%

Let $Y$ be the only element of $|\Oo_\FF(0,1)|$ that contains $C$. The exact sequence in Formula~\eqref{successione} gives $h^0(\Ii_E(1,x-2)) \le  h^0(\Ii_{E\cup C}(1,x-1))$ and equality holds if and only if $Y$ is in the base locus $\Bb$ of $|\Ii_{E\cup C}(1,x-1)|$. Call $C_1,\dots ,C_x$ the components of $E$ and $Y_i$ the only surface of bidegree $(0,1)$ containing $C_i$. By Remark~\ref{postc02} the irreducible surfaces $Y,Y_1,\dots ,Y_x$ are all different from each other.
For a general $A$ we get that every integer  $h^0(\Ii_E(1,x-2))$ is the same for every union of $x$ connected components of $A$. So if the inequality is false, then $\Bb$ contains the surface $Y\cup Y_1\cup \cdots \cup Y_x$ of bidegree $(0,x+1)$, which is a contradiction.
\end{proof}

\section{Surfaces of bidegree $(1,2)$ and $(1,3)$}
\label{sez-4}
In this section we specialize our study to the case of surfaces of bidegree $(1,2)$ and $(1,3)$.
In particular, we will prove Theorems~\ref{i2i1}, \ref{n6} and~\ref{n7}.

{Recall, from Formula~\eqref{p1}, that
for each  $A\in\Cc(n)$,
we have $\chi(\Ii_A(1,2)) =15-4n,$}
and hence, if $n\le 3$, we get $h^0(\Ii_A(1,2))> 0$. 
We also recall that a general $S\in |\Oo_{\FF}(1,2)|$ contains  a finite number of smooth conics and, thanks 
to Theorem~\ref{c01}, for every $B\in \Cc(2)$ we have $h^1(\Ii_B(1,1)) =0$.

%

%
%

\subsection{Surfaces of bidegree $(1,2)$ containing $0\le n\le 4$ twistor fibers}
In this section, we show the existence of a smooth surface of bidegree $(1,2)$ containing exactly
$0,1,2,3$ or $4$ twistor fibers. 
In order to analyze the space $|\Ii_A(1,2)|$ when $A$ is in $\Cc(n)$ (or in $\Tt(n)$), for $0\le n\le 4$,
we need some preliminary results. Note that the extremal case, when $n=4$, is treated differently.

We start by considering  $(1,2)$-surfaces containing {three disjoint smooth conics}. 

\begin{proposition}\label{n3}
Take $A\in \Cc(3)$ such that $h^0(\Ii_A(1,1)) >0$. Then 
\begin{enumerate}
\item there exists a curve $L$ of bidegree $(1,0)$ and a curve $R$ of bidegree $(0,1)$ such that $A\in\Cc(3,L)$ and $A\in\Cc(3,R)$ 
\item
 there is an irreducible element in $ |\Ii_A(1,2)|$;
\item $h^1(\Ii _A(1,2))=1$;
\item the base locus $\Bb$ of $|\Ii_A(1,2)|$ is $A\cup L\cup R$, where $L$ and  $R$ are the curves defined in $\textit{(1)}$.
\end{enumerate}
\end{proposition}

\begin{proof}
We start arguing as in Remark~\ref{remmmm}.
Thanks to Remark~\ref{m3}, every surface of bidegree $(1,0)$ or $(0,1)$ does not contain an element of $\Cc(2)$, so as $A\in\Cc(3)$, every element in 
$|\Ii_A(1,1)|$ is irreducible.
Proposition~\ref{c0} {(with $a=b=c=d=1$)} gives  that $h^0(\Ii_A(1,1)) =1$ and hence we set $|\Ii_A(1,1)|=\{M\}$ and by Formula~\eqref{p1} we compute $h^1(\Ii_A(1,1)) =2$.

We now prove the first statement.
Let $C$ be a connected component of $A$, set $B:=A\setminus C$ and denote by $X$ the only element of $|\Ii_C(1,0)|$. Performing the same construction that leads to Formula~\eqref{successione}, we have
$$
0\to\Ii_{B}(0,1)\to\Ii_{A}(1,1)\to\Ii_{A\cap X,X}(1,1)\to 0.
$$
Since  $h^0(\Ii_{B}(0,1)) =0$ and $h^0(\Ii_A(1,1)) =1$, the previous residual exact sequence gives 
\begin{equation}\label{formulage1}
h^0(\Ii_{A\cap X,X}(1,1)) \ge1.
\end{equation}
Thanks to Formula~\eqref{Hirz}, we have that $\Oo_X(1,1)\simeq\Oo_{F_1}(h+2f)$; moreover, recall from Remark~\ref{m3} that $C$ is identified with an element of $|\Oo_{F_1}(h+f)|$. 
Since $A\cap X$ is the union of $C$ and the two points $ B\cap X$, there exists a fiber $L\in |f|$ of the ruling of $F_1$ containing $ B\cap X$. 
Since $f (h+f) =1$, we have that $L$ meets $C$. So $L$ hits every connected component of $A$. If we take instead of $X$ the only element of $|\Ii_C(0,1)|$, we get the existence of $R$.

We now prove \textit{(2)}. We will show that there is an irreducible element in $ |\Ii_A(1,2)|$ by showing that the possible reducible cases
do not cover the entire family.
Remark~\ref{m3.1} shows that $A$ is not contained in a surface of bidegree $(1,2)$ with an irreducible component of
bidegree $(0,2)$. By Remark~\ref{m3}, a bidegree $(0,1)$
or $(0,1)$-surface does not contain any element of $\Cc(n)$, with $n\ge2$.
Thus, there are only finitely many elements of  $|\Oo_\FF(1,2)|$ with at least $3$ irreducible components.

Since $h^0(\Oo_{\FF}(0,1)) =3$,
the set of all reducible elements of $ |\Oo_{\FF}(1,2)|$ with an irreducible component of bidegree $(1,1)$ containing $A$ is isomorphic to $\PP^2$. Thus, in order to prove the existence of an irreducible element in $ |\Ii_A(1,2)|$, it is sufficient to show that $h^0(\Ii_A(1,2)) \ge 4$. 
But, using the exact sequence \eqref{sequenza}, this is equivalent to proving that $h^1(\Ii_A(1,{2})) \ge 1$, and the last inequality is true, {by Theorem \ref{c01}} because $\#(L\cap A)=3$. 

To prove \textit{(3)}, i.e. $h^1(\Ii_A(1,2))=1$, it is sufficient to prove that $h^1(\Ii _A(1,2))\le 1$. 
As before, take a connected component $C$ of $A$ and set $B= A\setminus C$. Let $Y$ be the only element of $|\Ii_C(0,1)|$. In the identification~\eqref{coeff} of $Y$ with $F_1$ we have
$$\Ii_{(B\cap Y)\cup C,Y}(1,2) \cong \Ii_{(B\cap Y)\cup C,F_1}(h+3f)\cong \Oo_{(B\cap Y),F_1}(2f).$$ Since $\#(B\cap Y)=2$ and $\Oo_{F_1}(2f)$ is globally generated, $h^1(F_1,\Ii_{B\cap Y,F_1}(2f)) \le 1$. We have $\Res_Y(A)=B$. Thanks to Theorem~\ref{c01}, we have $h^1(\Ii _B(1,1)) =0$, and the residual exact sequence of $Y$
$$
0\to\Ii_{B}(1,1)\to\Ii_{A}(1,2)\to\Ii_{(B\cap Y)\cup C,Y}(1,2)\to0,
$$
gives $h^1(\Ii_A(1,2))\le 1$. 

Finally, we discuss the base locus of $|\Ii_{A}(1,2)|$ {in order to} prove \textit{(4)}.
First of all, for any surface $S\in|\Ii_{A}(1,2)|$, we clearly have $A\subset S$. Moreover, since $\#(L\cap A)=3$ and $\#(R\cap A)=3$, then by B\'ezout, both curves are contained in $S$: in fact, thanks to Remark~\ref{rem1001} and  Formula~\eqref{tablechow}, the general intersection between a curve
of bidegree $(1,0)$ and $S$ consists of one point while the intersection of
a curve of bidegree $(0,1)$ and $S$ consists of two points (see also Remark~\ref{m0000}). Therefore,
$L\cup R\subset S$ and $A\cup L\cup R\subset\Bb$.

We now prove that $\Bb\subset A\cup L\cup R$. Fix $p\in \Bb\setminus (A\cup L\cup R)$. Take a connected component $C_i$, $i=1,2,3$, of $A$ and set $B_i:= A\setminus C_i$. Let $Y_i$ be the only element of $|\Oo_{\FF}(0,1)|$ containing $C_i$. By Proposition \ref{ee1} $B_i\cup L\cup R$ is the base locus of $|\Ii_{B_i}(1,1)|$. So there exists $S_i\in |\Ii_{B_i}(1,1)|$ such that $p\notin S_i$. If $p\notin Y_i$, then $p\notin \Bb$. Since $S_1\cap S_2\cap S_3 =L\cup R$, we can take $i\in \{1,2,3\}$  such that $p\notin Y_i$. So $\Bb =A\cup L\cup R$.
%
\end{proof}

{The following remark shows that if $A\in \Cc(3)$
satisfies  condition \textit{(1)} of Theorem \ref{n3}, then the existence of a $(1,1)$-surface containing $A$ is guaranteed. In particular
there exists a $(1,1)$-surface containing arbitrary triplets of collinear twistor fibers.}

\begin{remark}\label{n2}
Take $A\in \Cc(3)$ and assume the existence of curves $L$ of bidegree $(1,0)$ and $R$ of bidegree $(0,1)$ intersecting every connected component of $A$.
By adapting the proof of Lemma \ref{no2},
since $\#(L\cap A)= 3$ and $\#(R\cap A)= 3$,  
we have that $h^1(\Ii_A(1,1)) \ge 2$.
Thus $h^0(\Ii_A(1,1)) \ge 1$ and $A$ satisfies the assumptions of Proposition \ref{n3}.

We can be even more specific and say that if $A\in \Cc(3)$ (with no assumption about $L$ or $R$), then  $h^0(\Ii_A(1,1)) \le 1$   and if $|\Ii_A(1,1)|\ne \emptyset$, then the only element of $|\Ii_A(1,1)|$ is irreducible. This is true because, thanks to Remark \ref{m3}, every reducible element of $|\Oo_{\FF}(1,1)|$ contains at most $2$ {disjoint smooth conics}.

Note that if $A\in \Tt(3)$ and $L$ exists, then we can take $R:= j(L)$.  
So if $A\in \Tt(3)$ to get $h^0(\Ii_A(1,1)) >0$, it is sufficient to assume that $A\notin \Tt^*(3)$.
\end{remark}

The following lemma is a sort of inverse of the previous remark.
\begin{lemma}\label{u7.1}
Take $A\in \Cc^*(3)$. Then $h^0(\Ii_A(1,1))=0$ and $h^1(\Ii_A(1,1))=1$.
\end{lemma}

%

\begin{proof}
If $A\in\Cc(3)$, thanks to Formula~\eqref{p1}, $\chi(\Ii_A(1,1))=-1$.
So $h^0(\Ii_A(1,1))=0$ if and only if $h^1(\Ii_A(1,1))=1$. 
We assume that $h^0(\Ii _A(1,1)) \ne 0$ and will prove that $A\notin\Cc^{*}(3)$. Let $B\subset A$ be the union of $2$ connected components of $A$  and set $C:= A\setminus B$. Let $L$ and $R$ be the curves defined in Remark~\ref{LR} for $B\in \Cc(2)$. 
Take any element $D\in|\Ii_{B}(1,1)|$.
Since $\#(B\cap (L\cup R))=2$, $B\subset D$, and $D$ has bidegree $(1,1)$, then B\'ezout theorem implies $L\cup R\subset D$. 

By Theorem~\ref{c01} and Proposition~\ref{ee1}, we have $h^1(\Ii_B(1,1)) =0$, $h^0(\Ii_B(1,1)) =2$,
and the general element $M$ in $|\Ii_B(1,1)|$ is irreducible.
Since $h^0(\Ii_B(1,1)) =2$ and $M$ is general, $C\nsubseteq M$. Consider the following residual exact sequence:
\begin{equation}\label{equ8.1}
0 \to \Ii_C\to \Ii_A(1,1) \to \Ii_{B\cup (M\cap C),M}(1,1)\to 0
\end{equation}

Since $M\in |\Ii_B(1,1)|$ and $h^1(\Oo_{\FF})=0$, the exact sequence
$$0 \to \Oo_{\FF} \to \Ii _B(1,1)\to \Ii_{B,M}(1,1)\to 0$$
gives $h^0(M,\Ii_{B,M}(1,1)) =1$. Moreover, $h^{0}(\Ii_{C})=0$, and so the sequence~\eqref{equ8.1} and the assumption $h^0(\Ii_A(1,1)) \ge 1$ imply $h^0(M,\Ii_{B\cup (M\cap C),M}(1,1))\ge 1$. By Proposition~\ref{ee1}, the curve $A\cup L\cup R$ is
the base locus of $|\Ii_B(1,1)|$ and thus the base locus of $H^0(M,\Ii_{B,M}(1,1))$ is the curve $B\cup L\cup R$.
Since $B\cap C=\emptyset$, the degree $2$ scheme $C\cap M$ is contained in $L\cup R$. To get $A\notin \Cc^*(3)$, we must to prove that $C\cap L \ne  \emptyset$. It is sufficient to observe that $\deg (C\cap T)\le 1$ for any curve $T$ of bidegree $(0,1)$. {This is true because of  Remark~\ref{rem1001} and the fact that $C$ is the intersection of a surface of bidegree $(1,0)$ and a surface of bidegree $(0,1)$.}
\end{proof}


We now discuss the case of $A\in \Cc(2)$ contained in a smooth surface of bidegree $(1,1)$. In this case we will also prove smoothness for the general element in $ |\Ii_A(1,2)|$.

\begin{proposition}\label{ee2}
Take any $A\in \Cc(2)$ that is contained in a smooth element of $|\Oo_{\FF}(1,1)|$. Then we have
\begin{enumerate}
\item $h^1(\Ii _A(1,2)) =0$ and $h^0(\Ii_A(1,2))=7$;
\item the set $A\cup L$ is contained in the base locus $\Bb$ of $|\Ii_A(1,2)|$, where $L$ is the bidegree $(1,0)$ curve described in Remark~\ref{LR};
\item a general $S\in |\Ii_A(1,2)|$ is smooth.
\end{enumerate}
\end{proposition}

\begin{proof}
To prove \textit{(1)} it is sufficient to apply Corollary~\ref{cor1}, which gives $h^1(\Ii _A(1,2))=0$ and Formula~\eqref{p1}, which entails $h^0(\Ii_A(1,2))=7$. 

We now pass to point \textit{(2)}. 
Take $L$ and $R$ as in Remark~\ref{LR}. Since $\Oo_{\FF}(0,1)$ is globally generated, $\Bb \subseteq A\cup L\cup R$; moreover, thanks to Remark~\ref{m0000}, we also have $A\cup L\subseteq \Bb$. 

We are left to prove \textit{(3)}.
By Bertini's theorem $\mathrm{Sing}(S)\subseteq A\cup L\cup R$ for a general $S\in |\Ii_A(1,2)|$. Fix a smooth $M\in |\Ii_A(1,1)|$. Take a general $Y'\in |\Oo_{\FF}(0,1)|$. Since $Y'$ is general $L\cap Y'=\emptyset$ (and hence it is not singular at any $p\in L$). Thus, up to small deformation, we can say that $S$ (which is general) is smooth in a neighborhood of $L$. We are left to {exclude the case} $\mathrm{Sing}(S)\subseteq A\cup R$.
Fix $p\in A\cup R$ and let $2p$ be the $0$-dimensional scheme of $\FF$ defined by the ideal $\Ii^{2}_{p,\FF}$.  $S$ is singular at $p$ if and only if $S\in |\Ii_{2p\cup A}(1,2)|$. 
To complete our proof we need to prove that 
$$h^0(\Ii _{2p\cup A}(1,2)) = h^0(\Ii _A(1,2)) -2,$$
for all $p\in (A\cup R) \setminus A\cap R$ and that, for $p\in A\cap R$,
$h^0(\Ii_{2p\cup A}(1,2)) < h^0(\Ii _A(1,2))$.
These two statements give the thesis because $(A\cup R) \setminus A\cap R$ and
$A\cap R$ are $1$-dimensional and $0$-dimensional, respectively, and we are saying that
the set of bidegree $(1,2)$-surfaces containing $A$ and a singular point has
codimension $2$ in the first case and positive codimension in the second one.

%
%
%
%
%

Let us start by taking $p\in (A\cup R) \setminus A\cap R$. Since $p$ is a smooth point of $A\cup R$, $\deg (2p\cap (A\cup R))=2$.
Consider the exact sequence
\begin{equation}\label{eqqqqq}
0\to \Ii _{(A\cup R)\cup 2p}(1,2)\to \Ii _{A\cup R}(1,2)\to \Ii_{A\cup R}\otimes \Oo_{2p}(1,2)\to 0.
\end{equation}
Since $\deg (2p) =4$ and $A$ is smooth, we have $h^0(\Ii_{A\cup R}\otimes \Oo_{2p}(1,2))=2$ if $p\in A\cup R$ and $h^0(\Ii_{A}\otimes \Oo_{2p}(1,2))=4$ if $p\in R$.
Hence it is sufficient to prove that
$$h^1(\Ii _{(A\cup R)\cup 2p}(1,2)) =0.$$
 
First of all, assume that $p\in A\setminus R$.
Let $C$ be the connected component of $A$ containing $p$. Set the following notation $E:= A\setminus C$. 
Since $R$ is in the base locus of $\Ii_{A}(1,1)$ we have that $h^0(\Ii_{A}(1,1))=h^0(\Ii_{A\cup R}(1,1))$ (see~\cite[proof of Theorem 1.1]{abb}). Moreover, thanks to part \textit{(1)} and to \cite[Remark 4.3]{abb}, we have  $h^0(\Ii_{A\cup R}(1,1))=h^0(\Ii_{A}(1,1)) =h^0(\Ii_{E}(1,1)) -3$.
Thus $p$ is not in the base locus of $|\Ii_{E}(1,1)|$. 
Fix $M\in |\Ii_{E}(1,1)|$ such that $p\notin S$. 
Let $Y$ be the surface of $|\Oo_{\FF}(0,1)|$ containing $C$  and consider the residual exact sequence with respect to $Y$:
\begin{equation}\label{eqcc1}
0 \to \Ii_{E\cup p}(1,1) \to \Ii _{A\cup 2p}(1,2)\to \Ii_{{(E\cap Y)}\cup C\cup (2p\cap Y),Y}(1,2)\to 0.
\end{equation}

We now prove that 
\begin{equation}\label{1vanish}
h^1(\Ii _{E\cup p}(1,1)) =0.
\end{equation}

{Recall that $A=E\cup C$ and $p\in C$, hence we have the exact sequence
$$0 \to \Ii_{A}(1,1) \to \Ii _{E\cup p}(1,1)\to \Ii_{p,C}(2)\to 0.$$
Thanks to Theorem~\ref{c01} we have that  $h^1(\Ii _{A}(1,1)) =0$; on the other hand,  since $C$ is a smooth rational curve, we have
$h^1(\Ii_{p,C}(2))=h^1(\Oo_{C}(1))=0$, and this proves \eqref{1vanish}.
}

To conclude, it is now sufficient to prove that
\begin{equation}\label{2vanish}
h^{1}(\Ii_{(E\cap Y)\cup C\cup (2p\cap Y),Y}(1,2))=0.
\end{equation}

{
Note that 
$\Ii_{C,Y}(1,2)\cong \Oo_{F_1}(2f)\cong \Oo_Y(0,1)$,
hence, by~\cite[Remark 2.11]{abb} we know that $\Ii_{C,Y}(1,2)$ is very ample.
Therefore we get
$h^1(\Ii_{C\cup (2p\cap Y),Y}(1,2))=0$. Since $E\cap Y$ consists of one point, we conclude that \eqref{2vanish}
 holds.
}

Thus the exact sequence \eqref{eqcc1} gives $h^1(\Ii _{A\cup 2p,\FF}(1,2))=0$, completing the proof in the case $p\in A\setminus  R$.

Fix $p\in R\setminus (A\cap R)$ and ecall that we need to prove that $h^0(\Ii_{A\cup 2p}(1,2)) = 5$. Fix a general $Y'\in |\Ii_p(0,1)|$. Since $Y'$ is general, $R\nsubseteq Y'$ (and also $L\nsubseteq Y'$). Since $Y'$ is smooth,  $Y'\cap 2p = (2p,Y')$ is a scheme of degree $3$ and $\Res_{Y'}(2p) =\{p\}$.
As $p\in R$, we have that $h^0(\Ii_{A\cup \{p\}}(1,1)) =h^0(\Ii _A(1,1)) =2$. Thus, by the residual exact sequence of $Y'$ it is sufficient to prove that $$h^0(Y',\Ii_{(A\cap Y')\cup (2p,Y')}(1,2)) \le 3.$$
Since $\Oo_{Y'}(1,2)$ is very ample, we have $h^0(Y',\Ii_{(2p,Y')}(1,2)) =h^0(Y',\Oo_{Y'}(1,2)) -3 = 4$. So it is enough to prove that $A\cap Y'$ is not contained in the base locus, $\Bb'$, of $|\Oo_{(2p,Y')}(1,2)|$. 
In the identification between $Y'$ and $F_1$ we have $\Oo_{Y'}(1,2) \cong \Oo_{F_1}(h+3f)$. Let $N$ be the only element of $|f|$ that contains $p$. We have $N\cong \PP^1$ and $h^1(N,\Ii_{2p\cap N}(1,2)) =0$, but $N\subseteq \Bb'$.
Since $\Oo _{F_1}(h+2f)$ is very ample, $\Ii_p(h+2f)$ has only $p$ in its base locus. Thus $\Bb' =N$ and so, since $R\nsubseteq Y'$ (and also $L\nsubseteq Y'$), $\Bb'$ cannot contain both points of $A\cap Y'$.

The last case is $p\in A\cap R$. To prove our claim, i.e. that  $h^0(\Ii_{2p\cup A}(1,2)) < h^0(\Ii _A(1,2))$, it suffices to use \eqref{eqqqqq}.
Hence, a general $S$ is smooth.
\end{proof}

%
%
%

 The following result is analogous to Proposition~\ref{n3}, if we choose the
 conics to be twistor fibers.

\begin{proposition}\label{n2.1}
Take $A\in \Tt(3)$ such that $h^0(\Ii_A(1,1))=0$.
Then we have the following:
\begin{enumerate}
\item $h^1(\Ii _A(1,2))=0$ and hence $h^0(\Ii _A(1,2)) =3$;
\item there is an irreducible $S\in |\Ii_A(1,2)|$; 
\item the base locus of $|\Ii_A(1,2)|$ is contained {in} the union of $A$ and $3$ different curves of bidegree $(1,0)$;
\item for a sufficiently general $A$ (contained in a dense euclidean open subset of $\Tt(3)$), we can take a smooth $S\in |\Ii_A(1,2)|$.
\end{enumerate}
\end{proposition}

{ 
In the previous statement
we assume  
$A\in \Tt(3)$ such that $h^0(\Ii_A(1,1))=0$.
Note that, thanks to Lemma~\ref{primo-caso}, 
this implies that 
the conics in $A$ do not belong to any infinite family of twistor fibers contained in
a smooth $j$-invariant surface of bidegree $(1,1)$.}

\begin{proof}
We start with the proof of \textit{(1)}.
Fix a connected component $C$ of $A$ and call $D$ the only element of $|\Ii_C(0,1)|$. Set $B:= A\setminus C$. To get $h^1(\Ii _A(1,2))=0$, mimicking the proof of Proposition \ref{n3} it is sufficient to prove that $h^1(F_1,\Ii _{B\cap D}(2f)) =0$.
Suppose $h^1(F_1,\Ii _{B\cap D}(2f)) >0$, i.e. assume the existence of $T\in |\Oo_{F_1}(f)|$ containing the $2$ points $B\cap D$. Since $C\in |\Oo_{F_1}(h+f)|$, $C\cap T\ne \emptyset$. So $T$ hits every connected component of $A$.
Remark \ref{n2} gives $h^0(\Ii_A(1,1)) >0$, a contradiction. 

To prove \textit{(2)}, it is sufficient to show that the reducible cases
do not cover the whole $|\Ii_A(1,2)|$. In fact, reasoning as in the proof of Proposition~\ref{n3}, the only possible splittings are of the form $(1,0)+(0,1)+(0,1)$, which are in a finite number, or $(1,1)+(0,1)$, where the bidegree $(1,1)$ component 
contains $2$ connected components of $A$ and the remaining bidegree $(0,1)$ part is uniquely determined. Now, $h^0(\Ii _B(1,1))=2$, so, 
the set of all reducible elements of $|\Ii_{A}(1,2)|$ with an irreducible component of bidegree $(1,1)$ does not cover $|\Ii_A(1,2)|$.

We now prove \textit{(3)} and \textit{(4)}. Since $h^0(\Ii_A(1,1))=0$ and $A$ is $j$-invariant, neither $\pi _1(A)$ nor $\pi _2(A)$ has a triple point (both have $3$ double points).
Set $L_1\cup L_2\cup L_3:= \pi _2^{-1}(\mathrm{Sing}(\pi _2(A)))$ and $R_1\cup R_2\cup R_3:= \pi _1^{-1}(\mathrm{Sing}(\pi _1(A)))$. Since $\#(L_i\cap A)=\#(R_i\cap A)={2}$, $L_1\cup L_2\cup L_3$ are in the base locus of $|\Ii_A(1,2)|$ and each $L_i$ and each $R_i$ meets exactly $2$ connected components of $A$.

To prove the existence of a smooth element, it is sufficient to reason as in the proof of Proposition~\ref{ee2} case \textit{(3)}.
\end{proof}

{We are now ready to prove the first part of Theorem \ref{i2i1}.}

\begin{theorem}\label{terzo-I}
Fix $n\in \{0,1,2,3\}$. There is a smooth $S\in |\Oo_{\FF}(1,2)|$ which contains exactly $n$ twistor fibers. 
\end{theorem}
\begin{proof}
A general $S\in |\Oo_{\FF}(1,2)|$ contains only finitely many smooth conics. Since the set of all twistor fibers has real codimension $4$ in the space of all smooth conics, a general $S\in |\Oo_{\FF}(1,2)|$ contains no twistor fiber.

Now we prove the case $n=1$. Fix a twistor fiber $C$ and take a general $S\in |\Ii_C(1,2)|$. Assume that $S$ contains another twistor fiber, $E$. We have $h^1(\Ii_C(1,2)) =h^1(\Ii _{C\cup E}(1,2)) =0$ (Theorem \ref{c01} and Remark \ref{remarkCHI}). Thus $|\Ii_{C\cup E}(1,2)|$ is a $4$-codimensional complex projective subspace of $|\Ii_C(1,2)|$ (this is explained by the equality $h^{0}(\Ii_{C	\cup E}(1,2))=h^{0}(\Ii_{C}(1,2))-4$ contained in~\cite[proof of Theorem 1.1]{abb}). However $\Tt(1)$ is a real $4$-dimensional space. So a general $S\in |\Ii_C(1,2)|$ does not contain any other twistor fiber. 

{Note that $C$ is the base locus of $|\Ii_C(1,2)|$. By Bertini's theorem a general $S\in |\Ii_C(1,2)|$ is smooth outside $C$. Fix $p\in C$ and let $2p$ be the closed subscheme of $\FF$ with $(\Ii_p)^2$ as its ideal sheaf. Recall that $2p\subset S$ if and only if $p\in \mathrm{Sing}(S)$. Since $\dim C=1$, to get that $S$ is smooth, it is sufficient to prove that $h^0(\Ii_{2p\cup C}(1,2)) \le h^0(\Ii_C(1,2)) -2=9$. This follows from the proof of Proposition~\ref{ee2}} case \textit{(3)}.

The case $n=2$ is true by Proposition~\ref{ee2} with $\Tt(2)$ instead of $\Cc(2)$.

The case $n=3$ is true by Proposition~\ref{n2.1}.
\end{proof}

In the remainder of this section, we will construct a smooth $(1,2)$-surface 
containing $4$ twistor fibers. {The following lemma, in the case $d=2$, says 
 that
if an irreducible  $(1,2)$-surface contains $4$ disjoint smooth conics, then these conics are not general, because three of them must be collinear}.


\begin{lemma}\label{n5}
Let $d\ge 2$ and $A\in\Cc(d+2)$. If there is an irreducible $S\in |\Oo_{\FF}(1,d)|$, then
$A\notin\Cc^{*}(d+2)$.
\end{lemma}

\begin{proof}
We prove the lemma by induction on $d$.
We start with the case $d=2$.
Suppose that $A\in \Cc^*(4)$, i.e. there is no union $B$ of $3$ of the connected components of $A$ such that
$\#(L\cap B)=3$ for some curve $L$ of bidegree $(1,0)$. Fix a connected component $C$ of $A$ and set $B:= A\setminus C$. Call $Y$ the only element of $|\Ii_C(0,1)|$. Remark \ref{m3} gives $\Res_Y(A) =B$. By assumption and Lemma~\ref{u7.1}, $h^0(\Ii_{B}(1,1)) =0$.
Since $h^0(\Ii_A(1,2))\ne 0$, the residual exact sequence 
$$
0\to\Ii_{B}(1,1)\to\Ii_{A}(1,2)\to\Ii_{A\cap Y, Y}(1,2)\to0\,,
$$
gives $h^0(Y,\Ii _{A\cap Y,Y}(1,2))>0$ (otherwise $|\Ii_{A}(1,2)|=\emptyset$). The scheme $A\cap Y$ is the union of $C$ and the $3$ points $B\cap Y$. In the identification of $Y$ with $F_1$ the line bundle $\Oo_Y(1,2)$
goes to the line bundle $\Oo_{F_1}(h+3f)$ and $C$ goes to an element of $|h+f|$. Thus $h^0(F_1,\Ii_{B\cap Y,F_1}(2f)) >0$. Thus at least $2$ of the $3$ points $B\cap Y$ are in the same fiber $\hat L$ of the ruling $|f|$ of $F_1$. Since $\hat L\cap C\ne \emptyset$, $\hat L$ is a curve of bidegree $(1,0)$ that meets at least $3$ connected components of $A$. Let $B'$ be the union of $3$ components of $A$ that intersect $\hat L$. The curves $B'$ and $\hat L$ give a contradiction.

Now assume that the result is true for $d+1$. Note that as a by-product of the previous part, if $B\in\Cc^{*}(d+1)$, then $h^0(\Ii_B(1,d-1)) =0$.

Assume $A\in \Cc^*(d+2)$ and that there is an irreducible $S\in |\Oo_{\FF}(1,d)|$. 
Fix a connected component $C$ of $A$ and set $B:= A\setminus C$. Take a surface $Y$ of bidegree $(0,1)$ containing
$C$. By means of the sequence in Formula~\eqref{successione}, we either have $h^0(\Ii_B(1,d-1)) >0$ or $h^0(Y,\Ii _{A\cap Y,Y}(1,d)) >0$. Since $A\in \Cc^*(d+2)$, $B\in \Cc^*(d+1)$ and hence, thanks to the inductive assumption, we have $h^0(\Ii_B(1,d-1)) =0$. The scheme $A\cap  Y$ is the union of $C$ and the scheme $B\cap Y$ with $A\cap B\cap Y=\emptyset$.
Up to the identification of $Y$ and $F_1$ we have $\Oo_Y(1,d)(-C) \cong \Oo_{F_1}(df)$. Since $Y$ has bidegree $(0,1)$, every connected component of $B$ is either contained in $Y$ or it intersects transversely $Y$ at a unique point. By Remark \ref{m3}, the set $B\cap Y$ is formed by $d+1$ points. Thus $\Ii_{A\cap Y,Y}(1,d) \cong \Ii_{B\cap Y}(df)$. We saw that $h^0(Y,\Ii _{A\cap Y,Y}(1,d)) >0$, and this is true if and only if there are $u_{1},\dots u_{d+1}\in B\cap Y$ and $F\in |f|$ such that that $u_{i}\ne u_{j}$, for $i\neq j$ and $\{u_{1},\dots,u_{d+1}\}\subset F$. The set $F\cap C$ is a unique point, $o$, and $o\notin \{u_{1},\dots,u_{d+1}\}$, because $B\cap C=\emptyset$. The curve $F$ has bidegree $(0,1)$ and therefore $A\notin \Cc^*(d+2)$, a contradiction.
\end{proof}

%
%
%

Thanks to the previous result, if an irreducible $(1,2)$-surface contains $4$ disjoint smooth conics, then
they are in special position. We now show that if these $4$ conics are twistor fibers, then
their position is very special. 
We start by introducing the following notation.
For $n\ge 4$, we denote by $\Cc(n)^-$  the set of the elements $A\in\Cc(n)$ for which there exists a bidegree $(1,0)$ curve $L$ such that $A\in\Cc(n,L)$. {The set $\Cc(n)^-$ parametrizes the families of $n$ collinear disjoint smooth conics.} 
For $n\ge 4$ we also write
$\Tt(n)^-:= \Tt(n)\cap \Cc(n)^-$. The families $\Cc(n)^{-}$ and $\Tt(n)^{-}$ are Zariski closed in $\Cc(n)$ and $\Tt(n)$, respectively. 

{The following lemma shows that if an irreducible $(1,2)$-surface contains $4$ twistor fibers, then they are all collinear.} 
\begin{lemma}\label{bo1}
Take an irreducible $S\in |\Oo_{\FF}(1,2)|$ containing $A\in \Tt(4)$. Then $A\in \Tt(4)^-$.
\end{lemma}

\begin{proof}
Suppose there exists an irreducible $S\in |\Oo_{\FF}(1,2)|$ containing $A\in \Tt(4)$. By Lemma \ref{n5} there exists a union $B$ of $3$ of the connected components of $A$ such that
$B\in\Tt(3)\setminus\Tt^{*}(3)$, i.e., there exists a bidegree $(1,0)$ curve $L$,
such that $B\in\Tt(3,L)$ and thus, thanks to Remark~\ref{remmmm} $h^0(\Ii_B(1,1)) >0$. 
However, the same remark tells us that $h^0(\Ii _B(1,1))=1$, $h^1(\Ii _B(1,1))=2$, and that the only element $M$ of $|\Ii_B(1,1)|$ is irreducible.

As usual, set $C:= A\setminus B$.  As in Remark~\ref{n2} since $L$ of bidegree $(1,0)$
meets every connected component of $B$, then $R:= j(L)$, of bidegree $(0,1)$, do the same.

Thanks to Remark~\ref{m0000} we get $B\cup L\cup R\subset M$.
Since $S$ and $M$ are irreducible, thanks to Lemma~\ref{uc1}, the one-dimensional scheme $S\cap M$ has bidegree $(5,4)$. Since $B\cup L\cup R$ has bidegree $(4,4)$, then $C\nsubseteq M$. Let $Y$ be only element
of $|\Ii_C(0,1)|$. Since $B\subset M\cup Y$, then $M\cup Y\in |\Ii_{A}(1,2)|$. Moreover, since $S$ is irreducible, then $S\neq M\cup Y$, and hence $h^0(\Ii_A(1,2))\ge 2$, i.e. $h^1(\Ii _A(1,2)) \ge 3$. Since $h^1(\Ii_B(1,1))=2$, the residual exact sequence 
$$
0\to\Ii_{B}(1,1)\to\Ii_{A}(1,2)\to\Ii_{A\cap Y, Y}(1,2)\to 0,
$$
 gives $h^1(Y,\Ii _{A\cap Y,Y}(1,2)) >0$.
As in the proof of Lemma \ref{n5} we obtain the following inequality $h^1(F_1,\Ii _{B\cap Y}(2f)) >0$, i.e. there exists a curve $\hat L\in |f|$ of bidegree $(1,0)$ which intersects at least $2$ of the connected components of $B$. 

Call $B'$ the union of $2$ of the connected components of $B$
intersecting $\hat L$. Since $\hat L\cap C\ne \emptyset$ and $B'\cup C$ is $j$-invariant, every connected component of $B'\cup C$ meets $j(\hat L)$. Remark \ref{n2}, Proposition \ref{n3} and B\'ezout imply the existence of an irreducible surface $M'$ of bidegree $(1,1)$
containing $B'\cup C\cup \hat L\cup j(\hat L)$. Since $B'\subset M'$, $\hat L$ and $j(\hat L)$ contain at least $2$ points of $M'$, then $B'\cup \hat L\cup j(\hat L)\subset M$. 
But by Remark~\ref{LR} there is a unique curve of bidegree $(1,0)$ intersecting two different smooth conics,
so $\hat L=L$ and both $L$ and $j(L)$ intersect every connected component of $B$. Thus $L$ intersects each connected component of $A$, i.e. $A\in \Cc(4)^{-}$.
\end{proof}

As a byproduct of the proof of the previous result, we get the following lemma,
which  says that there are {infinitely many} irreducible $(1,2)$-surfaces containing
$4$ {collinear} twistor fibers.

\begin{lemma}\label{bo2}
Take $A\in \Tt(4)^-$ and assume that $A$ is not contained in a surface of bidegree $(1,1)$. Then $\dim |\Ii _A(1,2)| =1$ and  $|\Ii _A(1,2)|$ contains exactly $4$ reducible elements of $|\Oo_{\FF}(1,2)|$.
\end{lemma}

\begin{proof}
Let $L$ be the curve of bidegree $(1,0)$ that intersects every connected component of $A$. Since every connected component of $A$ is $j$-invariant, $j(L)$ intersects every connected component of $A$. In the proof of Lemma~\ref{bo1} we showed that $h^0(\Ii _A(1,2))\ge 2$. 
From the lines of that proof, it can be deduced that
only $4$ elements in $|\Ii_A(1,2)|$ are reducible and they are all obtained by fixing a connected component $C$ of $A$ and taking the union of the unique surface $M_C$ of bidegree $(1,1)$ containing $A\setminus C$ and the unique surface $Y_C$ of bidegree $(0,1)$ containing $C$. To complete the proof, it is sufficient to show that $h^0(\Ii _A(1,2))\le 2$. Take a connected component $C$ of $A$ and consider the residual exact sequence
\begin{equation}\label{eq=a1}
0 \to \Ii _C(0,1)\to \Ii_A(1,2)\to \Ii_{M_C\cap A,M_C}(1,2)\to 0.
\end{equation}
We have $h^0(\Ii_C(0,1))=1$, because the intersection of $2$ different elements of $|\Oo_Y(0,1)|$ is a curve of bidegree $(1,0)$. Thus by \eqref{eq=a1} to conclude the proof it is sufficient to prove that the image $\Vv$ of $H^0(\Ii _A(1,2))$ in $H^0(M_C,\Ii_{M_C\cap A,M_C}(1,2))$ has dimension at most $1$. Bez\'out gives that $A\cup L\cup j(L)$ is contained in the base locus of $|\Ii_{B,M_C}(1,2)|$. Every {$D\in |\Vv|$} has bidegree $(5,4)$ as a curve of $\FF$ 
and thus a general $D\in |\Vv|$ is the union (counting multiplicities as divisors of the smooth surface $M_C$) of $A\cup L\cup j(L)$ and a curve $E$ of bidegree $(1,0))$ as a curve of $\FF$.
Recall that $M_C$ is the blow up of $\PP^{2}$ at $3$ non collinear points and that these $3$ exceptional divisors are the only curves of $M_C$ with bidegree $(1,0)$. Since $M_C$ has only finitely many curves of bidegree $(1,0)$, $D$ is the same for all non-zero elements of $\Vv$ and hence $\dim \Vv =1$.
\end{proof}

{The following result completes the proof of Theorem \ref{i2i1}.}
\begin{theorem}\label{aaa1}
There exist irreducible $S\in |\Oo_{\FF}(1,2)|$ containing exactly $4$ twistor fibers and for any such $S$ and $A\in \Tt(4)$ with $A\subset S$, there exists a curve $L$
of bidegree $(1,0)$ intersecting each connected component of $A$. Furthermore, $h^0(\Ii_A(1,2))=2$ and each $S\in |\Ii_A(1,2)|$ is singular along $L$.
\end{theorem}

\begin{proof}
By Theorem \ref{n6} no irreducible surface of bidegree $(1,2)$ contains at least $5$ twistor fiber. The curve $L$ exists by Lemma \ref{bo1}.
Now we reverse the construction. We start with $A\in \Tt(4,L)^-$. Let $2L$ be the closed subscheme of the ``double line''. To prove that every $S\in |\Ii_A(1,2)|$ is singular at every point of $L$, it is suffices
to prove that $h^0(\Ii_A(1,2)) =h^0(\Ii_{A\cup 2L}(1,2))$. Lemma \ref{bo2} gives  $h^0(\Ii _A(1,2)) = 2$. Hence, it is sufficient to prove that $h^0(\Ii_{A\cup 2L}(1,2))>1$.
For any connected component $C$ of $A$, let $M_C$ be the only surface of bidegree $(1,1)$ containing $A\setminus C$ and let $Y_C$ be the only surface of bidegree $(0,1)$. Since $C\cap L\ne \emptyset$,
$L\cap Y_C\ne \emptyset$. Since $Y_C$ has bidegree $(0,1)$ and $L$ bidegree $(1,0)$, we get $L\subset Y_C$. Thus $L\subseteq M_C\cap Y_C$ and hence $|\Ii_{A\cup 2L}(1,2)|$ contains at least the $4$ reducible elements of $|\Ii_A(1,2)|$. Hence
$h^0(\Ii_{A\cup 2L}(1,2))>1$.
\end{proof}

%
%
%
%
%

\subsection{Non existence results for surfaces of bidegree $(1,2)$ and $(1,3)$} \label{last-subsection}
In this last section, we prove our   last two main results, i.e. Theorems~\ref{n6} and~\ref{n7}.

For any $A\in \Tt(n)^-$, $n\ge 4$, let us call $L$ and $R:=j(L)$ the curves of bidegree $(1,0)$ e $(0,1)$ respectively, intersecting all the connected components of $A$.

In view of our goal, we need to discuss the reducibility of some surfaces containing a certain number of twistor fibers.
First of all, fix an integer $n\ge 2$, take $B\in \Tt(4)$ such that $h^0(\Ii _B(1,1)) >0$ and call $M$ the unique (see e.g. Remark~\ref{n2}) surface of bidegree $(1,1)$ containing $B$. Since every element of $\Cc(1)$ is contained in an element of $|\Oo_{\FF}(0,1)|$ for each $E\in \Tt(n-1)$
there exists a reducible element $W\in |\Oo_{\FF}(1,k)|$, union of $M$ and $n-1$ surfaces of bidegree $(0,1)$ such that $B\cup E\subset W$.
The following lemma is a kind of inverse of this remark. Moreover, it will also be a key tool in the last two proofs.

\begin{lemma}\label{bo5}
If $d\ge2$ and $A\in \Tt(d+3)^-$ are such that $h^0(\Ii _A(1,d)) >0$, then  every element of $|\Ii _A(1,d)|$ has an irreducible component $M$ of bidegree $(1,1)$ containing at least $4$ connected components of $A$.
In particular, for any  $n\ge d+3$, there is no irreducible $S\in |\Oo_{\FF}(1,d)|$ containing $A\in \Tt(n)^-$.
\end{lemma}

\begin{proof}
{In order to prove the last statement, it is sufficient to do the case $n=d+3$,} and thus it is sufficient to prove the first assertion.  

{We use induction on $d\ge2$. Let us first assume $d=2$.}
Take $A\in \Tt(5)^-$ and let $L$ and $j(L)$ be the curves of bidegree $(1,0)$ and $(0,1)$ intersecting all the connected components of $A$. 
Fix a connected component $C$ of $A$ and set $B:= A\setminus C$. Since $C\cap L\ne \emptyset$, the curve $C\cup L$ is a connected and nodal curve of bidegree $(2,1)$ with arithmetic genus $0$. 
Hence $h^0(\Oo_{C\cup L}(0,1)) =2$. Thus there is $Y\in |\Ii_{C\cup L}(0,1)|$ and such a $Y$ is unique. Since any two smooth conics of $Y$ meet, no component of $B$ is contained in $Y$.
Therefore, $B\cap Y$ is formed by $4$ points of $L\setminus (L\cap C)$. Recall that $\Oo_Y(1,2)\cong \Oo_{F_1}(h+3f)$ and that $C\in |\Oo_{F_1}(h+f)|$ and
thus $\Ii _{A\cap Y,Y} \cong \Ii_{B\cap L,Y}(2f)$. Since every element of $|f|$ contains a unique point of $L$ we have that $h^0(D,\Ii_{A\cap Y,Y}(1,2))=0$.
The residual exact sequence of $Y$
$$
0\to\Ii_{B}(1,1)\to\Ii_{A}(1,2)\to\Ii_{A\cap Y, Y}(1,2)\to 0,
$$
gives an isomorphism $\phi: H^0(\Ii_B(1,1))\to H^0(\Ii_A(1,2))$. If $h^0(\Ii_B(1,1))=0$, then $h^0(\Ii _A(1,2)) =0$.
Now suppose $h^0(\Ii_B(1,1))\ne 0$. The isomorphism $\phi$ says that every $W\in|\Ii_A(1,2)|$ has $Y$ as an irreducible component, say $W =Y\cup W_1$ with $W_1\in |\Ii_B(1,1)|$, and hence we have the thesis. 

Now assume $d\ge 3$ and use induction on $d$. By reasoning as in the base case,
take $A\in \Tt(d+3)^-$ and use the  exact sequence
$$
0\to\Ii_{B}(1,d-1)\to\Ii_{A}(1,d)\to\Ii_{A\cap Y, Y}(1,d)\to 0,
$$
to prove that $h^0(Y,\Ii_{A\cap Y,Y}(1,d))=0$ and thus that 
there is an isomorphism $\phi: H^0(\Ii_B(1,d-1))\to H^0(\Ii_A(1,d))$. 
Now, again, if $h^0(\Ii_B(1,d-1))=0$, then $h^0(\Ii _A(1,d)) =0$.
So we assume $h^0(\Ii_B(1,d-1))\ne 0$. The isomorphism $\phi$ says that each $S\in|\Ii_A(1,d)|$ has $Y$ as an irreducible component, i.e. $S =D\cup S_1$ with
$S_1\in |\Ii_B(1,d-1)|$. The inductive assumption says that $S_1$ has an irreducible component $M$ of bidegree $(1,1)$ containing at least $4$ components of $B$.
\end{proof}

We now have all the ingredients to prove Theorems~\ref{n6} and~\ref{n7}. 
First, we prove that no irreducible surface of bidegree $(1,2)$ contains 5 twistor fibers.

\begin{proof}[Proof of Theorem~\ref{n6}]
Suppose there exists an irreducible $S\in |\Oo_\FF(1,2)|$ containing $A\in \Tt(5)$. 
Lemma \ref{bo1} shows that for any union $A'\subset A$ of $4$ components of $A$, there exists a union $A''\subset A'$ of $3$ connected components intersecting some $L$ of bidegree $(1,0)$.
Let $L$ be a curve of bidegree $(1,0)$ intersecting the maximal number, $z$, of components of $A$. Obviously $z\ge 3$. By Lemma \ref{bo5}  to get a contradiction, it is sufficient to prove that { $z\ge5$.} 

Assume $z\in \{3,4\}$. Take any order $C_1,\dots ,C_5$ of the connected components of $A$ and set $B_i:= \pi _1(C_i)$, $1\le i\le 5$. 
Each $B_i$ is a line of $\PP^2$. Since any two conics contained in an element of $|\Oo_{\FF}(1,0)|$ meet, $B_1,\dots ,B_5$ are $5$ different lines of $\PP^2$. 
For any $i<j<h$ there exists a curve $T$ of bidegree $(1,0)$ intersecting $C_i$, $C_j$ and $C_h$ if and only if $B_h$ contains the point $B_i\cap B_j$ and in this case $L =\pi _1^{-1}(C_i\cap C_j)$.Without loss of generality, we can assume that $L$ meets $C_1,\dots ,C_z$.

\quad (a) Assume $z=3$ and hence $B_{1}\cap B_{2}\in B_{3}$.
Applying Lemma \ref{bo1} to $C_1\cup C_2\cup C_4\cup C_5$ we have one of the following mutually exclusive relations:
$$
\begin{matrix}
B_1\cap B_2\cap B_4\ne \emptyset, & & B_1\cap B_2\cap B_5\ne \emptyset,\\
B_1\cap B_4\cap B_5\ne \emptyset,& & B_2\cap B_4\cap B_5\neq\emptyset.
\end{matrix}
$$
Since $B_1\cap B_2\in B_3$ and $z=3$, we can exclude the first two cases, i.e. we have $$B_1\cap B_2\cap B_4= B_1\cap B_2\cap B_5= \emptyset.$$ 
Thus, either $B_1\cap B_4\cap B_5\ne \emptyset$ or $B_2\cap B_4\cap B_5\ne \emptyset$. Exchanging if necessary
$C_1$ and $C_2$ we may assume $B_1\cap B_4\cap B_5\ne \emptyset$, i.e. $B_4\cap B_5\in B_1$, and hence $B_2\cap B_4\cap B_5= \emptyset$. Since $B_2\cap B_4\cap B_5=\emptyset$, applying Lemma \ref{bo1} to $C_2\cup C_3\cup C_4\cup C_5$
we have one of the following mutually exclusive  relations 
$$
\begin{matrix}
B_2\cap B_3\cap B_4\ne \emptyset, & & B_2\cap B_3\cap B_5\ne \emptyset,\\
B_3\cap B_4\cap B_5\ne \emptyset. & & 
\end{matrix}
$$
 Since $B_4\cap B_5\in B_1$ and $z=3$, $B_3\cap B_4\cap B_5= \emptyset$. Since $B_2\cap B_3 \in B_1$ and $z=3$, $B_2\cap B_3\cap B_5= \ B_2\cap B_3\cap B_4= \emptyset$, a contradiction.

\quad (b) Assume $z=4$. Since $C_1\cap L\ne \emptyset$, the curve $C_1\cup L$ is a connected and nodal curve of arithmetic genus $0$ and bidegree $(2,1)$. Thus $h^0(\Oo_{C_1\cup L}(0,1))=2$. 
Thus there is $Y\in |\Ii_{C_1\cup L}(0,1)|$. Since $S$ is irreducible, $Y$ is not an irreducible component of $S$. So the residual exact sequence of $Y$
$$
0\to\Ii_{B}(1,1)\to\Ii_{A}(1,2)\to\Ii_{A\cap Y, Y}(1,2)\to 0,
$$ 
gives $h^0(Y,\Ii _{A\cap Y,Y}(1,2)) \ne 0$.
Up to the isomorphism of $Y$ and $F_1$ we have $L =h$, $C_1\in |\Oo_{F_1}(h+f)|$ and $\Oo_Y(1,2)\cong \Oo_{F_1}(h+3f)$. Since $(A\setminus C_1)\cap C_1=\emptyset$, $\Ii _{A\cap Y,Y}(1,3) \cong \Ii_{(A\setminus C_1)\cap D,D}(2f)$.
Since $L$ meets $C_1,\dots ,C_4$, $(A\setminus C_1)\cap D$ contains a set $F\subset L$ such that $\#F=3$. Since every element of $|f|$ contains a unique point of $L$, $h^0(Y,\Ii _{A\setminus C)\cap Y,Y}(2f))=0$, a contradiction.
\end{proof}

We now conclude  with the proof of Theorem~\ref{n7}, which concerns surfaces of bidegree $(1,3)$.

%
\begin{proof}[Proof of Theorem \ref{n7}:]
Assume the existence of $A\in \Tt(6)$ and of an irreducible $S\in |\Oo_{\FF}(1,3)|$ containing $A$. By Lemma~\ref{bo5} to get a contradiction it is sufficient to prove the existence of a curve $L$ of bidegree $(1,0)$ such that all  the components of $A$
intersect $L$. By Lemma \ref{n5} for any union $A'\subset A$ of $5$ components of $A$ there is a union $A''\subset A'$ of $3$ connected components intersecting some $L$ of bidegree $(1,0)$.
Let $L$ be a curve of bidegree $(1,0)$ intersecting the maximal number, $z$, of components of $A$. We have that $z\ge 3$. 
Therefore, by Lemma~\ref{bo5} it is sufficient to prove that {$z\ge6$.} 
Assume then that $z\le 5$. We will now exclude all the cases $z=3,4,5$.

For any connected component $C$ of $A$, Lemma~\ref{n5} tells us that there exists a curve $L$ of bidegree $(1,0)$ which intersects at least $3$ connected components of $A\setminus C$. In particular there is an irreducible $M\in |\Oo_{\FF}(1,1)|$ containing at least $3$ components of $A$ (see Remark~\ref{n2}). Note that $j(M)=M$. We can take $M$ with the additional condition that it contains the maximal number $e$ of components of $A$.
Let $E$ be the union of the components of $A$ contained in $M$. So $3\le e\le z\le 5$.

Since each twistor fiber is $j$-invariant, $j(L)$ hits every connected component of $E$. B\'ezout gives $L\cup j(L)\subset M$ and $L\subset S$. If $e\ge 4$, B\'ezout gives $j(L)\subset S$.
However, the one-dimensional cycle $M\cap S$ has bidegree $(7,5)$ and thus $e\le 4$. 
Set $\Sigma:= S\cap M$ (as a scheme-theoretic intersection). 
Since the one-dimensional scheme $\Sigma $ is the complete intersection of $\FF$ with $2$ very ample divisors, $h^0(\Oo_{\Sigma}) =1$. Set $F:= A\setminus E$.

\quad (a) Assume $e=4$. So $E\cup L\cup j(L)\subset \Sigma$. Since $E\cup L\cup j(L)$ has bidegree $(5,5)$ and $h^0(\Oo_{\Sigma}) =1$, $\Sigma$ is the union of $E\cup L\cup j(L)$ and a multiple structure on $L$.
Note that $\Sigma\in |\Oo_{\FF}(1,3)|$ and that $\Sigma$ contains $E\cup j(L)$ with multiplicity $1$ and  $L$ with multiplicity $3$ (as divisors of the smooth surface $M$). Since $\Sigma$ has multidegree $(7,5)$, $\Sigma =3L\cup j(L)\cup E$. Note that $\Sigma$ contains the degree $4$ zero-dimensional scheme $F\cap M$. Since $F\cap E =\emptyset$, $F\cap (j(L)\cup L)\ne \emptyset$. Thus at least one irreducible component, $T$, of $F$ meets $L\cup j(L)$.
Since $j(T)=T$, $T\cap L\ne \emptyset$. Thus $z=5$. Let $C$ be a component of $E$. Since $C\cap L\ne \emptyset$, $C\cup L$ is a connected and nodal curve of bidegree $(2,1)$ with arithmetic genus $0$.
Thus $h^0(\Oo _{C\cup L}(0,1)) =2$. So there is $Y\in |\Ii_{C\cup L}(0,1)|\ne 0$. Since $S$ is irreducible, $Y$ is not an irreducible component of $S$. Thus the residual exact sequence of $Y$ gives $h^0(Y,\Ii _{A\cap Y,Y}(1,3)) \ne 0$.
Up to the isomorphism of $Y$ and $F_1$ we have $L =h$, $C\in |\Oo_{F_1}(h+f)|$ and $\Oo_Y(1,3)\cong \Oo_{F_1}(h+4f)$. Since $(A\setminus C)\cap C=\emptyset$, $\Ii _{A\cap Y,Y}(1,3) \cong \Ii_{(A\setminus C)\cap Y,Y}(3f)$.
Since $z=5$, $(A\setminus C)\cap Y$ contains a set $H\subset L$ such that $\#H=4$. Since every element of $|f|$ contains a unique point of $L$, $h^0(Y,\Ii _{A\setminus C)\cap Y,Y}(3f))=0$, a contradiction.

\quad (b) Assume $e=3$.  Fix a connected component $C$ of $E$ and set $B:= A\setminus C$. Set $\{Y\}:= |\Ii_C(0,1)|$. As in step (a), we have that $L\subset Y$. The following exact sequence
\begin{equation}\label{eqnn1}
0\to \Ii_B(1,2)\to \Ii_A(1,3)\to \Ii_{C\cup (B\cap Y),Y}(1,3)\to 0
\end{equation}
is the residual exact sequence of $Y$. Since $Y$ is not an irreducible component of $S$, we have $h^0(Y,\Ii _{C\cup (B\cap Y),Y}(1,3)) >0$. As in step (a), we have $\Ii _{C\cup (B\cap Y),Y}(1,3) \cong \Ii_{B\cap Y,F_1}(3f)$.  We now have two possibilities: either $h^0(\Ii _B(1,2))=0$ or $h^0(\Ii _B(1,2))>0$.

\quad (b1) Let us assume for the moment that $h^0(\Ii _B(1,2))=0$. Then $h^0(Y,\Ii _{C\cup (B\cap Y),Y}(1,3)) \ge 2$ and $h^1(Y,\Ii _{C\cup (B\cap Y),Y}(1,3)) \ge 3$. Up to the identification of $Y$ and $F_1$ we have $\Ii_{C,Y}(1,3)\cong \Oo_{F_1}(3f)$.
So the $5$ points $B\cap Y$ give at most one condition to the linear system $|\Oo_{F_1}(3f)|$. Thus there is $J\in |\Oo_{F_1}(f)|$ such that $B\cap Y\subset J$. Note that $J$ is a curve of bidegree $(1,0)$. The maximality of the integer $e$ gives a contradiction.

\quad (b2) Assume that  $h^0(\Ii_B(1,2)) >0$. By Theorem~\ref{n6} every surface containing B is reducible, say $M_1\cup Y$ with $M_{1}$ irreducible of bidegree $(1,1)$ containing at least 4 components of B. Thus $e\ge4$, a contradiction.
%
\end{proof}

%
%

\end{document}